\documentclass[12pt]{article}
\usepackage{graphicx}
\usepackage{amsmath}
\usepackage{amssymb}
\usepackage{theorem}
\numberwithin{equation}{section}

\newtheorem{thm}{Theorem}[section]

\newtheorem{prop}[thm]{Proposition}

{\theorembodyfont{\rmfamily}
\newtheorem{defn}[thm]{Definition}

\newtheorem{rmk}[thm]{Remark}
}

\newcommand{\cJ}{\mathcal{J}}
\newcommand{\cM}{\mathcal{M}}

\newcommand{\eps}{{\varepsilon}}

\newcommand{\qed}{\hfill \mbox{\raggedright \rule{.07in}{.1in}}}
 
\newenvironment{proof}{\vspace{1ex}\noindent{\bf
Proof}\hspace{0.5em}}{\hfill\qed\vspace{1ex}}
\newenvironment{pfof}[1]{\vspace{1ex}\noindent{\bf Proof of
#1}\hspace{0.5em}}{\hfill\qed\vspace{1ex}}

 \newcommand{\tY}{{\widetilde{Y}}}
 \newcommand{\tF}{\widetilde{F}}
 \newcommand{\tW}{\widetilde{W}}

\newcommand{\R}{{\mathbb R}}

\newcommand{\Z}{{\mathbb Z}}

\newcommand{\N}{{\mathbb N}}
\newcommand{\E}{{\mathbb E}}
\newcommand{\PP}{{\mathbb P}}

\newcommand{\tZ}{{\widetilde Z}}
\newcommand{\sgn}{\operatorname{sgn}}

\newcommand{\SMALL}{\textstyle}
\newcommand{\BIG}{\displaystyle}

\usepackage{moreverb}
\usepackage[colorinlistoftodos]{todonotes}
\newcounter{gagcomment}

\title{Simulation of non-Lipschitz stochastic differential equations driven by $\alpha$-stable noise: a method based on deterministic homogenisation}

\author{
	Georg A. Gottwald\thanks{School of Mathematics and Statistics, University of Sydney, Sydney 2006 NSW, Australia}
	\and
	Ian Melbourne\thanks{Mathematics Institute, University of Warwick, Coventry, CV4 7AL, UK}
}

\date{20 April 2020}

\begin{document}

\maketitle

%%%%%%%%%%%%%%%%%%%%%%%%%%%%%%%%%%%%%%%%%%%%%%%%%%

\begin{abstract}
We devise an explicit method to integrate $\alpha$-stable stochastic differential equations (SDEs) with non-Lipschitz coefficients. To mitigate against numerical instabilities caused by unbounded increments of the L\'evy noise, we use a deterministic map which has the desired SDE as its homogenised limit. Moreover, our method naturally overcomes difficulties in expressing the Marcus integral explicitly. We present an example of an SDE with a natural boundary showing that our method respects the  boundary whereas Euler-Maruyama discretisation fails to do so. As a by-product we devise an entirely deterministic method to construct $\alpha$-stable laws.
\end{abstract}

%%%%%%%%%%%%%%%%%%%%%%%%%%%%%%%%%%%%%%%%%%%%%%%%%%

\section{Introduction}
Stochastic differential equations (SDEs) are frequently used to capture model uncertainty in as diverse areas as finance, engineering, biology and physics. The noise driving the SDE is often heuristically introduced based on the experience of the modeller.   In certain cases, the driving noise is derived by means of functional limit theorems, eg.\ in the context of fast-slow systems or weakly coupled systems of distinguished degrees of freedom with an infinite reservoir \cite{GivonEtAl04}. Recently, SDEs driven by non-Gaussian noise, in particular by L\'evy processes which involve discontinuous jumps of all sizes, have attracted attention. Anomalous diffusion and L\'evy flights are found in diverse systems ranging from biology \cite{DieterichEtAl08,WeigelEtAl11,Fernandez03,GetzSaltz08,BartumeusLevin08}, chemistry \cite{SanchoEtAl04,ReuveniEtAl10}, fluid dynamics  \cite{SolomonEtAl1994} to climate science \cite{Ditlevsen99,SeoBowman00,HuberEtAl01}. 

We consider here SDEs of the form
\begin{align}
dZ = a(Z)\, dt + b(Z) \diamond dW
\label{eq:SDE0}
\end{align}
where $Z\in\R^d$ and $W$ denotes an $m$-dimensional L\'evy process. 
The diamond denotes that stochastic integrals are to be interpreted in the Marcus sense~\cite{Marcus81}.  (We refer to~\cite[p.~272]{Applebaum} for a discussion of the Marcus integral.)
The drift term $a:\R^d\to\R^d$ and diffusion term $b:\R^d\to\R^{d\times m}$ are assumed to be smooth but we are particularly interested in situations where they are not globally Lipschitz on $\R^d$.  As is standard in the literature on numerical analysis of SDEs, we use the word  ``non-Lipschitz'' when referring to terms that  are smooth but not globally Lipschitz.

The Marcus interpretation for the stochastic integral in~\eqref{eq:SDE0} is known to arise naturally in SDEs driven by L\'evy processes, since it is the integral that transforms under the usual laws of calculus~\cite[Theorem~4.4.28]{Applebaum}.  As such, it plays the same role for L\'evy processes as the Stratonovich integral for Brownian motion.
Accordingly, if an SDE driven by a L\'evy process is to model a physical system and is therefore derived as a rough limit of an inherently smooth underlying microscopic dynamical system, then one anticipates that the driving noise should be interpreted in the sense of Marcus. 
Indeed, for deterministic fast-slow systems converging to an SDE driven by a L\'evy process, the Marcus interpretation has been proved to prevail 
by~\cite{ChevyrevEtAl19, GottwaldMelbourne13c}.
(However, if more than one time-scale is involved, then the noise may be Marcus, It\^o or neither \cite{LiEtAl14,ChechkinPavlyukevich14}.)
\\

 The numerical simulation of SDEs of the form (\ref{eq:SDE0}) poses three challenges: (i) the Marcus integral, (ii) non-Lipschitz drift and diffusion terms, (iii) nonexplicit nature of the densities for the increments of $W$.
These challenges are unrelated and typically require separate attention; some are better understood than others. We present here a method which naturally addresses all three problems simultaneously. Before we present the ideas behind our method, we discuss the particular problems of each challenge. 

(i) Marcus integrals $\int_0^t b(Z(s))\diamond dW(s)$ are well-defined but involve cumbersome expressions and sums over infinitely many jumps \cite{Applebaum,ContTankov,ChechkinPavlyukevich14}. 
In particular, the situation is quite different from the It\^o-Stratonovich correction where one can pass between It\^o and Stratonovich integrals by modifying the drift term. When numerically approximating Marcus integrals, several methods exist to discretise the integral (see \cite{AsmussenGlynn,Grigoriu09,Falsone18} and references therein). These methods typically use that a symmetric L\'evy process can be approximated as a sum of a compound Poisson process and a Brownian motion \cite{AsmussenRosinski01}. However, for nonsymmetric L\'evy processes, Brownian motion is not able to capture the skewness of the small jumps, presenting further difficulties for the numerical simulation of the corresponding Marcus SDE.

(ii) A well-known problem arises when numerically simulating SDEs with non-Lipschitz drift and diffusion terms. To illustrate why this may present a problem, consider the SDE with constant diffusion and non-Lipschitz drift term,  $dZ = -Z^3 dt + dW$ where $W$ is Brownian motion.
(The nature of the noise is not relevant in the following argument, just that the increments are unbounded).
  Its Euler-Maruyama discretisation \cite{Maruyama55,Milstein,KloedenPlaten} is given by
\begin{align*}
Z_{n+1}=Z_n-Z_n^3\, \Delta t + \sqrt{\Delta t}\,\Delta W_n,
\end{align*}
with normally distributed increments $\Delta W_n$. Since such increments are unbounded, for each fixed time step $\Delta t$ there is a non-zero probability that increments are so large as to lead to a numerical instability whereby the $Z_n$ explode alternating in sign. In particular, Euler-Maruyama fails to strongly converge in the mean-square sense and also fails to weakly converge to solutions of the SDE \cite{HutzenthalerEtAl11}. Recently, several numerical methods were designed to overcome the problem of non-Lipschitz drift terms \cite{HighamEtAl02,MilsteinTretyakov05,HutzenthalerEtAl12,Sabanis13,TretyakovZhang13,DareiotisEtAl16,Mao16,KumarSabanis17,KellyLord16}. However, to the best of our knowledge, no methods have been designed to treat with the presence of non-Lipschitz diffusion terms.

(iii)  The increments of Brownian motion are normally distributed with density function given by the well-known Gaussian formula. For the increments of L\'evy processes, the densities are not given explicitly in general. Various methods have been devised that numerically generate the desired probability densities \cite{ChambersMallowsStuck76}.  Of the three issues we have mentioned, this is the only one that could be said to be completely resolved, though even here there is the question of combining it with methods dealing with issues (i) and (ii).
\\

To bypass the cumbersome direct approximation of the Marcus integral and the difficulties associated with nonsymmetric L\'evy processes mentioned above, and to avoid the problem of unbounded noise increments, we propose an entirely deterministic method, based on homogenisation, to integrate SDEs of the form (\ref{eq:SDE0}). In particular, we use that a discrete deterministic fast-slow system reduces in the limit of infinite time scale separation to an SDE \cite{GottwaldMelbourne13c,KellyMelbourne16,CFKMZsub,ChevyrevEtAl19}. In the case of intermittent fast dynamics, the resulting SDE is driven by a L\'evy process, moreover the noise is of Marcus type \cite{GottwaldMelbourne13c,ChevyrevEtAl19}. We employ statistical limit theorems to design an explicit fast intermittent map and an explicit observable of the fast dynamics that yields $\alpha$-stable increments with user-specified values of the driving L\'evy process. The jumps of the L\'evy process are approximated by many small jumps generated by the fast dynamics. Since the fast dynamics evolves on a compact set, these increments are naturally bounded, which mitigates numerical instability caused by the non-Lipschitz terms.\\

The paper is organised as follows. We review the definitions of $\alpha$-stable laws in Section~\ref{sec:stable} and provide algorithms to deterministically generate $\alpha$-stable laws and numerical illustrations of its accuracy. Section~\ref{sec:Levy} contains the corresponding material for L\'evy processes.
Section~\ref{sec:SDE} constitutes the main result of our work and introduces the numerical method to integrate SDEs driven by a L\'evy process using deterministic homogenisation.  Two examples of scalar SDEs are used to illustrate the method. In Example 1, our results are in line with Euler-Maruyama discretisation (with taming).  However, Example 2 has a natural boundary at $Z=0$ which is treated correctly by our method but not by the Euler-Maruyama method.
The proofs for our methods are provided in Section~\ref{sec:proofs}. We conclude with a discussion and an outlook in Section~\ref{sec:disc}. 
  
%%%%%%%%%%%%%%%%%%%%%%%%%%%%%%%%%%%%%%%%%%%%%%%%%%

\section{Generating $\alpha$-stable laws}
\label{sec:stable}

In this section, we show how to generate stable laws deterministically.
In Subsection~\ref{sec:def}, we review the definitions.
In Subsection~\ref{sec:Thaler}, we describe the Thaler map which will be used to generate the fast intermittent dynamics.
Our numerical algorithm for generating stable laws is presented in Subsection~\ref{sec:stablegen}.  Numerical illustrations of its accuracy are given in Subsection~\ref{sec:stablenum}.

\subsection{Definition of stable laws}
\label{sec:def}

A random variable $X$ is called a {\em (strictly) stable law} if there exist 
constants $b_n>0$ such that
independent copies $X_1,X_2,\ldots$ of $X$ satisfy 
\[
\SMALL b_n^{-1}\sum_{j=1}^n X_j=_d X\quad\text{for all $n\ge1$.}
\]
Stable laws are completely classified, see~\cite{Feller66,Applebaum}.
If $\E X^2<\infty$, then $X$ is normally distributed, $X\sim N(0,\sigma^2)$ where $\sigma^2=\E X^2$, and we can take $b_n=n^{1/2}$.
We are interested here in the case $\E X^2=\infty$.

There are various parameters (with various notational conventions).
The most important is the stability parameter or scaling exponent $\alpha\in(0,2]$.  
A suitable choice of $b_n$ is then given by $b_n=n^{1/\alpha}$.
The case $\alpha=2$ corresponds to the normal distribution described above, while  $\alpha=1$ corresponds to the Cauchy distribution which is a special case that we do not consider in this paper.
We restrict attention to the remaining stable laws
$X_{\alpha,\eta,\beta}$ whose characteristic function is given by
\[
\E(e^{itX_{\alpha,\eta,\beta}})=\exp\Bigl\{-\eta^\alpha|t|^\alpha\Bigl(1-i\beta\sgn(t)\tan\frac{\alpha\pi}{2}\Bigr)\Bigr\},
\]
where $\alpha\in(0,1)\cup(1,2)$, $\eta>0$ and $\beta\in[-1,1]$.
Such stable laws satisfy  $\E|X_{\alpha,\eta,\beta}|^p<\infty$ for $p<\alpha$ and $\E|X_{\alpha,\eta,\beta}|^\alpha=\infty$. 
In the case $\alpha\in(1,2)$ the stable law is centered, i.e.\ $\E X_{\alpha,\eta,\beta}=0$. 
A stable law is called {\em one-sided} (or {\em totally skewed}) if $\beta=\pm1$ and {\em symmetric} if $\beta=0$.

\begin{rmk} \label{rmk:eta}
It follows from the definitions that
$X_{\alpha,c\eta,\beta}=
cX_{\alpha,\eta,\beta}$ for $c>0$ and
$X_{\alpha,\eta,-\beta}=-
X_{\alpha,\eta,\beta}$.
\end{rmk}

%%%%%%%%%%%%%%%%%%%%%%%%%%%%%%%%%%%%%%%%%%%%%%%%%%

\subsection{The Thaler map}
\label{sec:Thaler}

In this section, we show how to generate all stable laws of the type $X_{\alpha,\eta,\beta}$ with $\alpha\in(0,1)\cup(1,2)$, $\eta>0$, $\beta\in[-1,1]$, using a deterministic dynamical system. 
In particular we shall use observables of maps introduced in the study of intermittency by Pomeau \& Manneville \cite{PomeauManneville80}. 
Particularly convenient for our purposes is the family of maps $T:[0,1]\to[0,1]$ considered by Thaler~\cite{Thaler80,Thaler00} 
\begin{align} 
\label{eq:Thaler}
 Tx=x\Bigl(1+\Bigl(\frac{x}{1+x}\Bigr)^{\gamma-1}-x^{\gamma-1}\Bigr)^{1/(1-\gamma)}\bmod1.
\end{align}
Here, $\gamma\in[0,1)\cup(1,\infty)$ is a real parameter.
Let $x^\star\in(0,1)$ be the unique solution to the equation
\begin{align} \label{eq:x1}
{x^\star}^{1-\gamma}+(1+{x^\star})^{1-\gamma}=2.  
\end{align}
There are two branches defined on the intervals $[0,x^\star]$, $[x^\star,1]$. 
See Figure~\ref{fig:ThalerMap} for a depiction of the Thaler map.

\begin{figure}
	\centering
	\includegraphics[height=19pc,width=19pc]{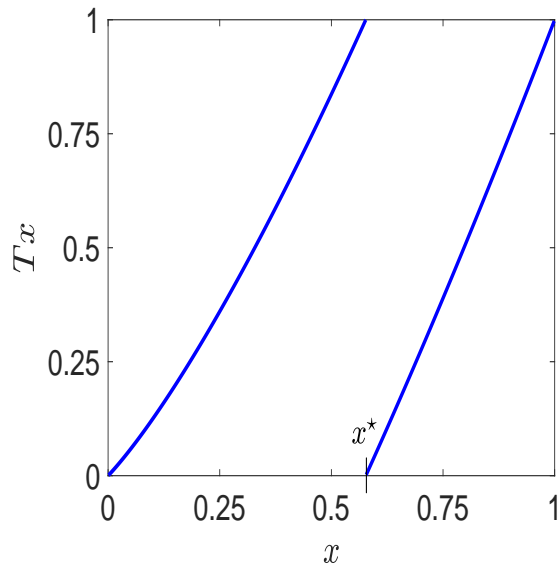}		
	\caption{Thaler map for $\gamma=0.625$ showing the two branches with domains $[0,x^\star]$ and $[x^\star,1]$ where $x^\star\approx 0.577$.}
	\label{fig:ThalerMap}
\end{figure}
\begin{rmk}  A useful alternative expression for the Thaler map is
\[
Tx=(x^{1-\gamma}+(1+x)^{1-\gamma}-1)^{1/(1-\gamma)}\bmod1.
\]
From this it is clear that $T$ has two increasing full branches and that
$x^\star$ is given by the formula mentioned above.
\end{rmk}

Unlike other intermittent maps such as the 
map $x\mapsto x+x^2\bmod1$ considered by Manneville~\cite{Manneville80} or
the Liverani-Saussol-Vaienti map \cite{LiveraniEtAl99}, the Thaler map allows for analytic expressions, both for the map and the invariant density. In particular, for each $\gamma\in [0,1)$ there exists a unique invariant probability density
 $\dfrac{1-\gamma}{2^{1-\gamma}}h$ where
\begin{align} 
\label{eq:h}
 h(x)=x^{-\gamma}+(x+1)^{-\gamma}.
\end{align}
For $\gamma>1$, the density $h$ is still well-defined and invariant, but it is nonintegrable so the corresponding invariant measure is infinite. For $\gamma=0$, the Thaler map reduces to the uniformly expanding doubling map $Tx=2x\bmod1$ with $h\equiv1$ corresponding to Lebesgue measure  on the unit interval; here correlations decay exponentially. 
For $\gamma\in(0,1)$, the Thaler map is nonuniformly expanding with a neutral fixed point at $x=0$ and
correlations decay algebraically with rate $n^{-(\gamma^{-1}-1)}$ \cite{Hu04,Young99}.  
The rate $n^{-(\gamma^{-1}-1)}$ is sharp by~\cite{Gouezel04a,Sarig02}.
This slow down in the decay of correlation as $\gamma$ increases is caused by the trajectory spending prolonged times near the neutral fixed point $x=0$. 
Figure~\ref{fig:Thalerxn} shows a trajectory for $\gamma=0.625$ where one clearly sees the laminar dynamics near $x=0$. 

\begin{figure}
	\centering
	\includegraphics[width=19pc]{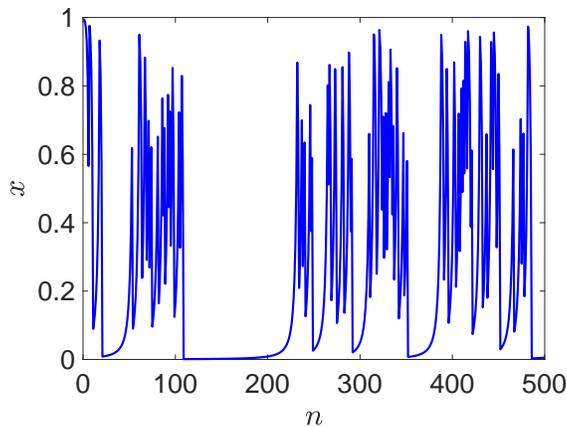}		
	\caption{Time series $x_n$ for the Thaler map with $\gamma=0.625$ corresponding to $\alpha=1.6$}
	\label{fig:Thalerxn}
\end{figure}

The above discussion shows that correlations are summable if and only if $\gamma<\frac12$,  leading to the following central limit theorem (CLT).
Let $v:[0,1]\to\R$ be a H\"older observable and suppose that $v$ has mean zero with respect to the invariant probability measure $d\mu=h\,dx$.  Define the Birkhoff sum $v_n=\sum_{j=0}^{n-1}v\circ T^j$ and the variance $\sigma^2\ge0$ (typically nonzero) via the Green-Kubo formula
$\sigma^2=\int v^2\,d\mu+2\sum_{n=1}^\infty \int v\,v\circ T^n\,d\mu$.
Regarding $n^{-1/2}v_n$ as a family of random variables on the probability space
$([0,1],\mu)$ (where the randomness exists solely in the initial condition $x_0\in [0,1]$ used to compute $n^{-1/2}v_n$) it follows from 
Liverani~\cite{Liverani96} that the CLT holds: $n^{-1/2}v_n\to_d N(0,\sigma^2)$.

For $\gamma\ge\frac12$, correlations are not summable and the CLT breaks down for observables with $v(0)\neq0$ that ``see'' the neutral fixed point at $x=0$. Heuristically the reason for this is that the 
Birkhoff sum $v_n$  experiences ballistic behaviour with almost linear behaviour in the laminar region near $x=0$ and the small jumps of size $v(0)$ accumulate into a single large jump incompatible with the CLT.
Indeed, Gouezel~\cite{Gouezel04} (see also~\cite{Zweimuller03}) proved that
for $\gamma\in (\tfrac12,1)$, the CLT is replaced by a one-sided stable limit law 
$n^{-\gamma}v_n\to_d X_{\alpha,\eta,\beta}$ 
with $\alpha=\gamma^{-1}$ and $\beta=\sgn v(0)$.

For $\gamma\ge1$, the density $h$ is not integrable and the Birkhoff sums $v_n$ (normalised) do not converge in distribution to a stable law.  However, the method in~\cite{Gouezel07} reduces via inducing~\cite{MT04} to an ``induced'' system on $Y=[x^*,1]$.  
A calculation using~\eqref{eq:x1} shows that $\int_{x^*}^1 h\,dx = \dfrac{2^{1-\gamma}-1}{1-\gamma}$ which is finite for all $\gamma\in[0,1)\cup(1,\infty)$.  
Hence we can define a probability measure $\mu_Y$ on $Y$ with
density $\dfrac{1-\gamma}{2^{1-\gamma}-1} h|_Y$.
For the induced system on the probability space $(Y,\mu_Y)$, convergence to stable laws was studied by~\cite{AaronsonDenker01} and holds in the full range $\gamma\in(\frac12,1)\cup(1,\infty)$.

To prove convergence to stable laws in this section and to L\'evy processes in Section~\ref{sec:Levy}, we use the induced system on $Y$, and hence are able to deterministically generate $\alpha$-stable random variables and processes for $\alpha\in(0,1)\cup(1,2)$. However, for our main application to SDEs in Section~\ref{sec:SDE}, we have to work with the full system on $[0,1]$ and hence our results there are restricted to $\alpha\in(1,2)$.

The aim in this section is to specify appropriate observables $v$ of the Thaler map leading to stable laws as limits in distribution.

%%%%%%%%%%%%%%%%%%%%%%%%%%%%%%%%%%%%%%%%%%%%%%%%%%

\subsection{Numerical algorithm for generating stable laws}
\label{sec:stablegen}

We begin by describing how to generate one-sided stable laws, i.e. those with $\beta=\pm 1$ where all jumps are in the same direction (positive or negative). 

Fix $\alpha\in(0,1)\cup(1,2)$ 
and consider the Thaler map (\ref{eq:Thaler}) with 
$\gamma=\alpha^{-1}$,
Define the set $Y=(x^\star,1]$ where $x^\star$ is as given in~\eqref{eq:x1}.
Starting with a randomly chosen initial condition $y_0\in Y$
(random with respect to the invariant density $h$ in~\eqref{eq:h} restricted to $Y$), we compute the iterates $T^k$ of the map $T$ noting the return times to $Y$.
More precisely, let $\tau_0\ge1$ be least such that $T^{\tau_0}y_0\in Y$.
Then let $\tau_1\ge1$ be least such that $T^{\tau_0+\tau_1}y_0\in Y$.
Inductively, once $\tau_0,\dots,\tau_{j-1}$ are defined, we let
$\tau_j\ge1$ be least such that 
$T^{\tau_0+\dots+\tau_j}y_0\in Y$.
Note that $\tau_0,\tau_1,\ldots$ is a sequence of random variables where the randomness originates from the choice of $y_0$.

Define 
\begin{align}
d_\alpha=\alpha^\alpha\frac{1-\gamma}{2^{1-\gamma}-1}g_\alpha,
\qquad
\ell_\alpha =\begin{cases} \hphantom{YYY} 0 & \alpha\in(0,1)
\\ (1-2^{\gamma-1})^{-1} & \alpha\in(1,2)\end{cases}, 
\label{eq:da}
\end{align}
where 
\begin{align} 
\label{eq:ga}
g_\alpha=\Gamma(1-\alpha)\cos\frac{\alpha\pi}{2}.
\end{align}

\begin{thm}   \label{thm:stable1}
Fix $\alpha\in(0,1)\cup(1,2)$.  Then
\[
n^{-\gamma} d_\alpha^{-\gamma}\SMALL (\sum_{j=0}^{n-1}\tau_j-n\ell_\alpha)\to_d X_{\alpha,1,1}\quad\text{as $n\to\infty$}.
\]
That is,
\[
\mu_Y\big\{y_0\in Y:n^{-\gamma} d_\alpha^{-\gamma}\SMALL (\sum_{j=0}^{n-1}\tau_j(y_0)-n\ell_\alpha)\le c\big\}\to \PP(X_{\alpha,1,1}\le c)\quad\text{as $n\to\infty$}
\]
for all $c\in\R$.
\end{thm}

\begin{rmk} By Remark~\ref{rmk:eta}, we can use Theorem~\ref{thm:stable1} to generate all one-sided
$\alpha$-stable laws 
$X_{\alpha,\eta,\pm1}=\pm \eta X_{\alpha,1,1}$.
\end{rmk}

We now extend to the case of general (two-sided) stable laws $X_{\alpha,\eta,\beta}$ with
$\alpha\in(0,1)\cup(1,2)$, $\eta>0$, $\beta\in[-1,1]$.
Again, we can suppose without loss that $\eta=1$.

Let $\tau_j$, $j\ge1$, be the sequence of random variables defined
in above.  Also, define the random variable
$\delta$ with
$\PP(\delta=\pm1)=\frac12(1\pm\beta)$ and let $\delta_j$, $j\ge0$, be a sequence of independent copies of $\delta$.

\begin{thm} \label{thm:stable2}
Fix $\alpha\in(0,1)\cup(1,2)$, $\beta\in[-1,1]$.  
Then
\[
\SMALL n^{-\gamma} d_\alpha^{-\gamma} (\sum_{j=0}^{n-1} \delta_j\tau_j-n\beta \ell_\alpha)\to_d X_{\alpha,1,\beta}\quad\text{as $n\to\infty$}.
\]
\end{thm}

Theorems~\ref{thm:stable1} and~\ref{thm:stable2} are proved in Section~\ref{sec:proofs}.

\subsection{Numerical results for stable laws}
\label{sec:stablenum}
We now illustrate that the algorithms described in Theorems~\ref{thm:stable1} and \ref{thm:stable2} are able to reliably construct $\alpha$-stable laws. The Thaler map $T$ is iterated for as many times as it takes to produce associated return times $\tau_0,\dots,\tau_{n-1}$ for some specified $n$.  
This data is then fed into
Theorems~\ref{thm:stable1} and \ref{thm:stable2}.
Note that the required number of iterates of $T$ is $\tau_0+\dots+\tau_{n-1}$ and
depends on the initial condition $y_0\in Y$, which is chosen randomly using the invariant density $h$ given by (\ref{eq:h}), restricted to $Y$.

In Figure~\ref{fig:Lawpdf}, we compare the results of our deterministic algorithm for approximating the probability density for $\alpha$-stable laws $X_{\alpha,\eta,\beta}$ with a direct numerical routine (we used the function \texttt{stblpdf} from the software package STABLE \cite{NolanStable}).  We take  $\alpha=1.6$, $\eta=0.5$ and $\beta=0$, $\beta=1$ and $\beta=-0.4$.  The two methods agree very well. The deterministically generated stable law was estimated from $50,000$ realisations (i.e.\ different initial conditions $y_0$) and we took $n=10,000$. To achieve data $\tau_0,\dots,\tau_{n-1}$ with the desired length $n=10,000$, the Thaler map was iterated for an average of $40,000$ times. The largest number of iterations needed for the realisations used here was more than $200,000$. 

Next we consider an example with $\alpha<1$. Figure~\ref{fig:Lawpdf2} shows the probability density for $\alpha$-stable laws with $\alpha=0.8$, $\eta=0.5$ and $\beta=0$, $\beta=1$ and $\beta=-0.4$. We used here $50,000$ realisations of data $\tau_0,\dots,\tau_{n-1}$ of length $n=10,000$ for $\beta=0$ and $\beta=-0.4$ and $n=50,000$ for $\beta=1$. Due to the higher probability to experience large jumps for $\alpha=0.8$ compared to $\alpha=1.6$, the number of iterations of the Thaler map needed to generate an induced time series of length $n$ is much larger. 
Here the Thaler map was iterated for an average of $2\times 10^6$ times for $\beta=0$ and $\beta=-0.4$ and for $10^7$ times for $\beta=1$. The largest number of iterations needed for the realisations used here was more than $140\times 10^6$ for $\beta=0$ and $\beta=-0.4$ and $230\times 10^6$ for $\beta=1$. 

\begin{figure}
	\centering
	\includegraphics[width=19pc]{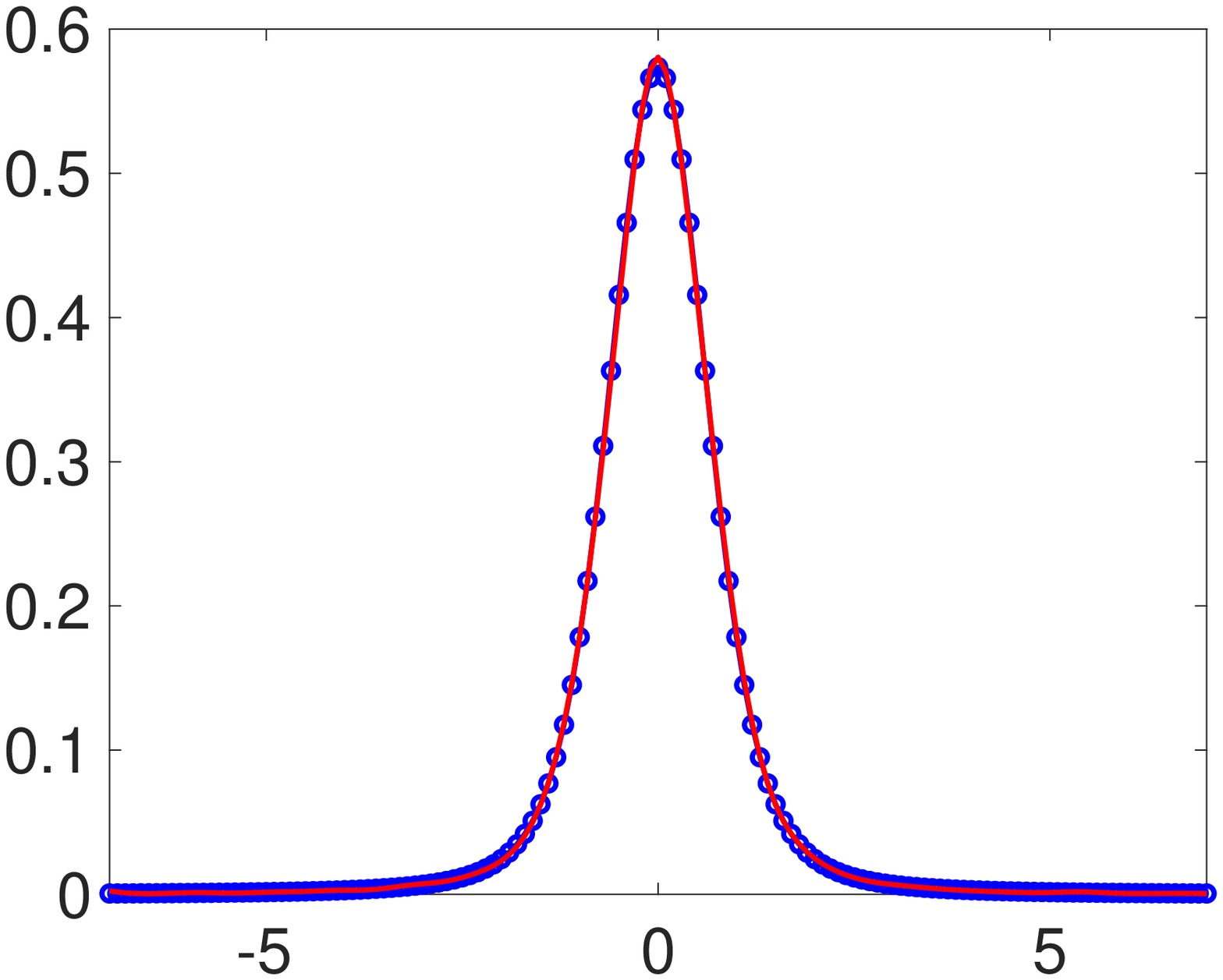}\\		
	\vspace{1mm}
	\includegraphics[width=19pc]{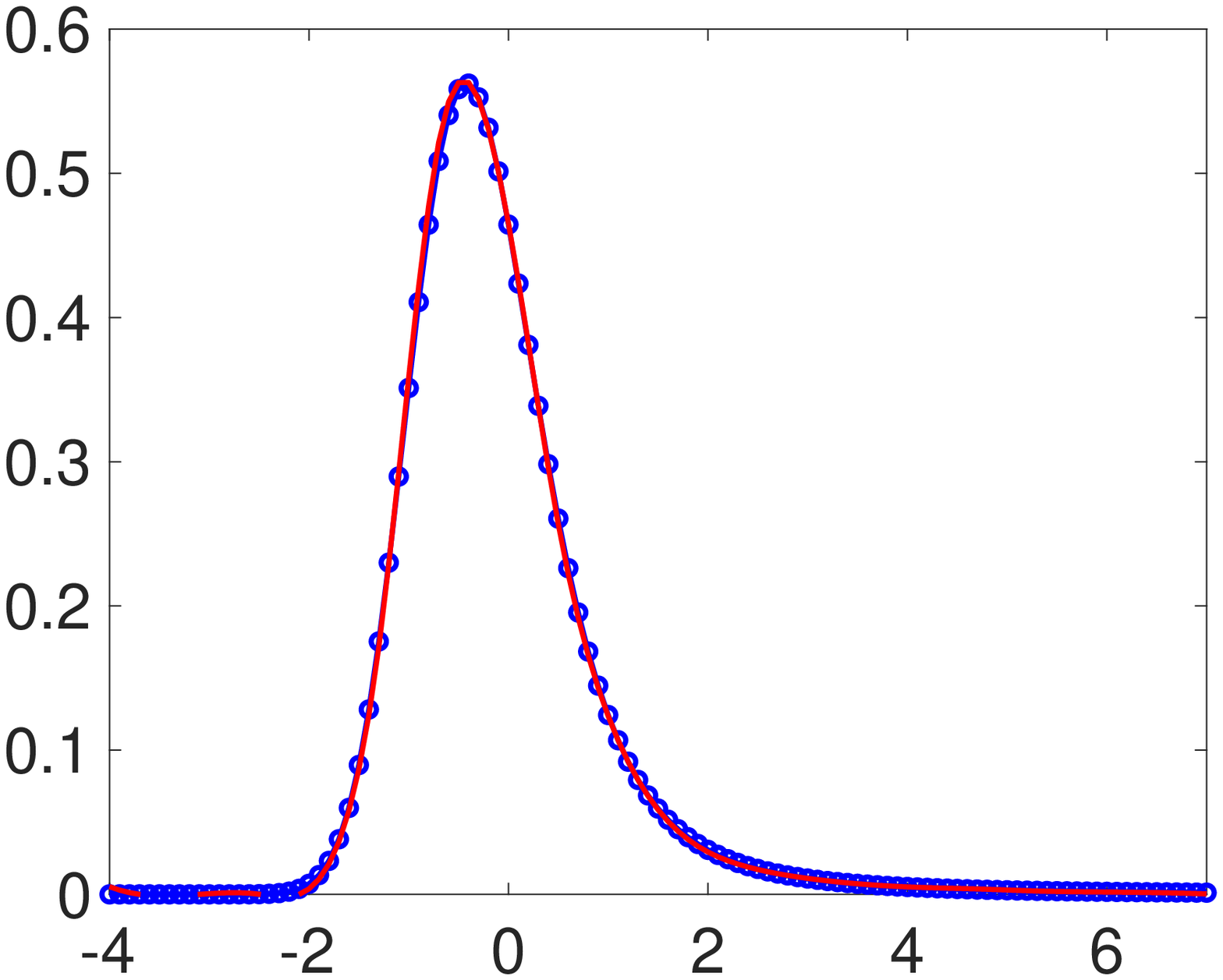}\\		
	\vspace{1mm}
	\includegraphics[width=19pc]{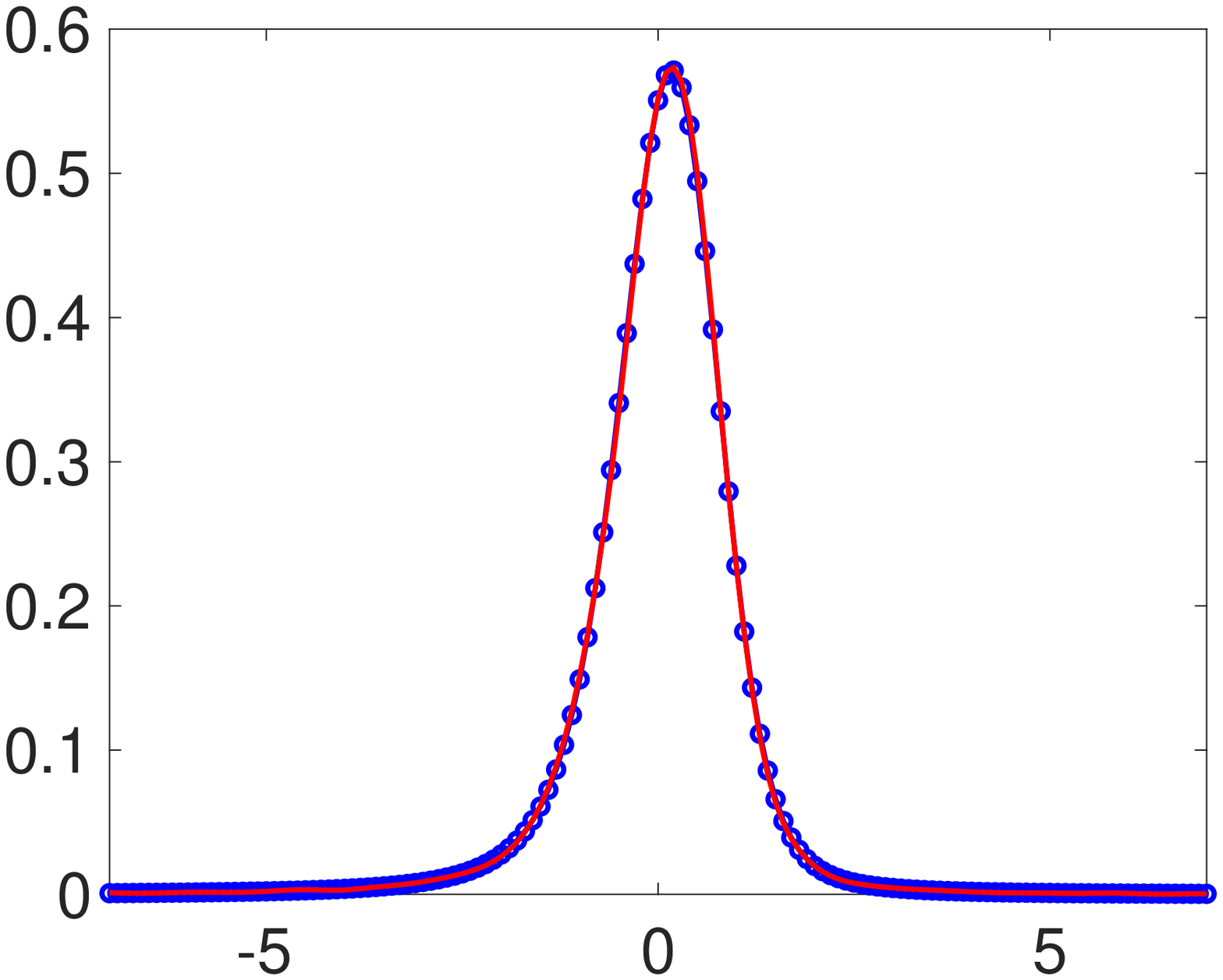}		
	\caption{Probability density functions for $\alpha$-stable laws $X_{\alpha,\eta,\beta}$ with $\alpha=1.6$, $\eta=0.5$ and (top): $\beta=0$, (middle): $\beta=1$ and (bottom): $\beta=-0.4$. The blue curve (open circles) uses the function \texttt{stblpdf} from the software package STABLE \cite{NolanStable}; the red continuous line shows the splined empirical histogram of the deterministic induced dynamics.}
	\label{fig:Lawpdf}
\end{figure}

\begin{figure}
	\centering
	\includegraphics[width=19pc]{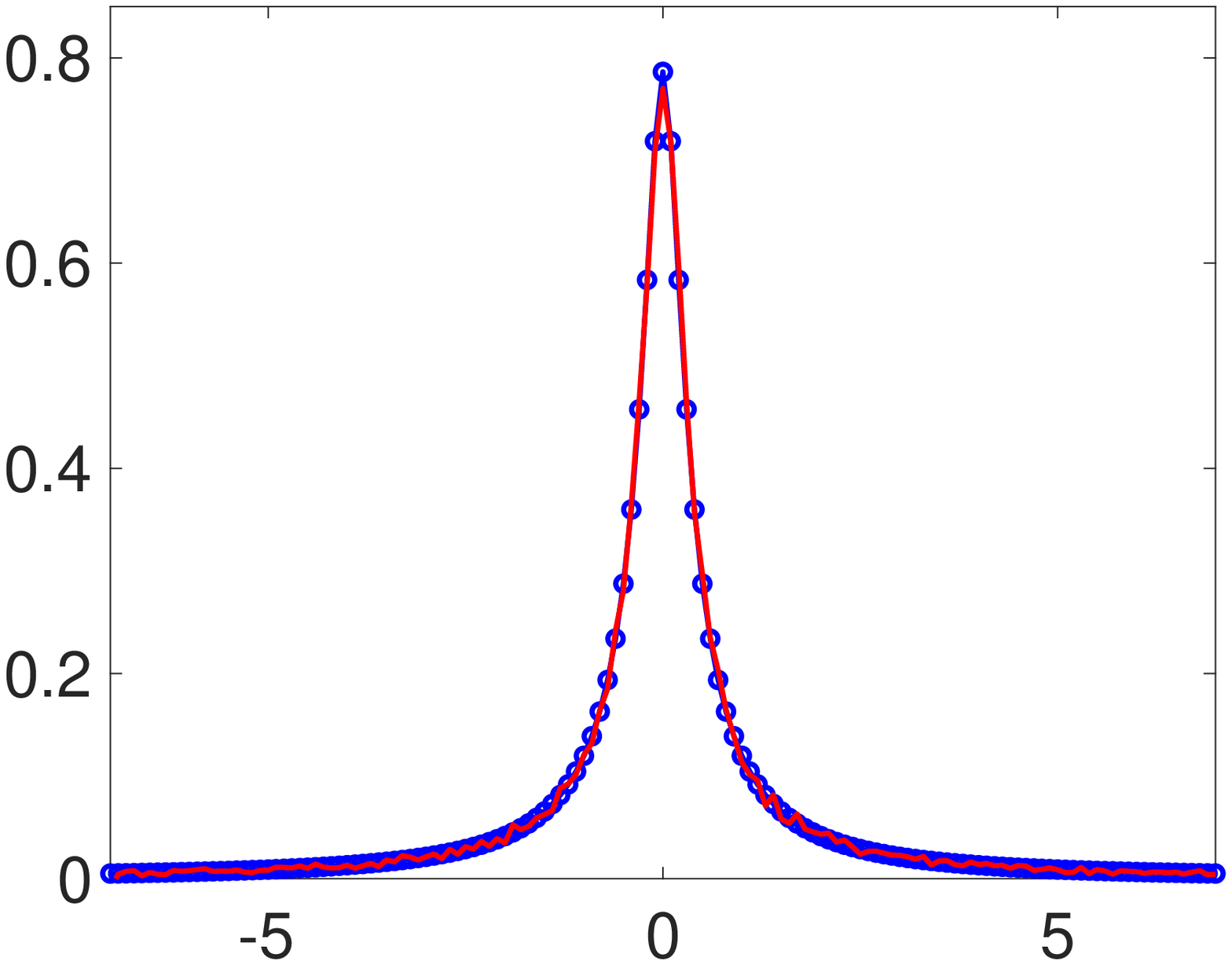}\\
	\vspace{1mm}
	\includegraphics[width=19pc]{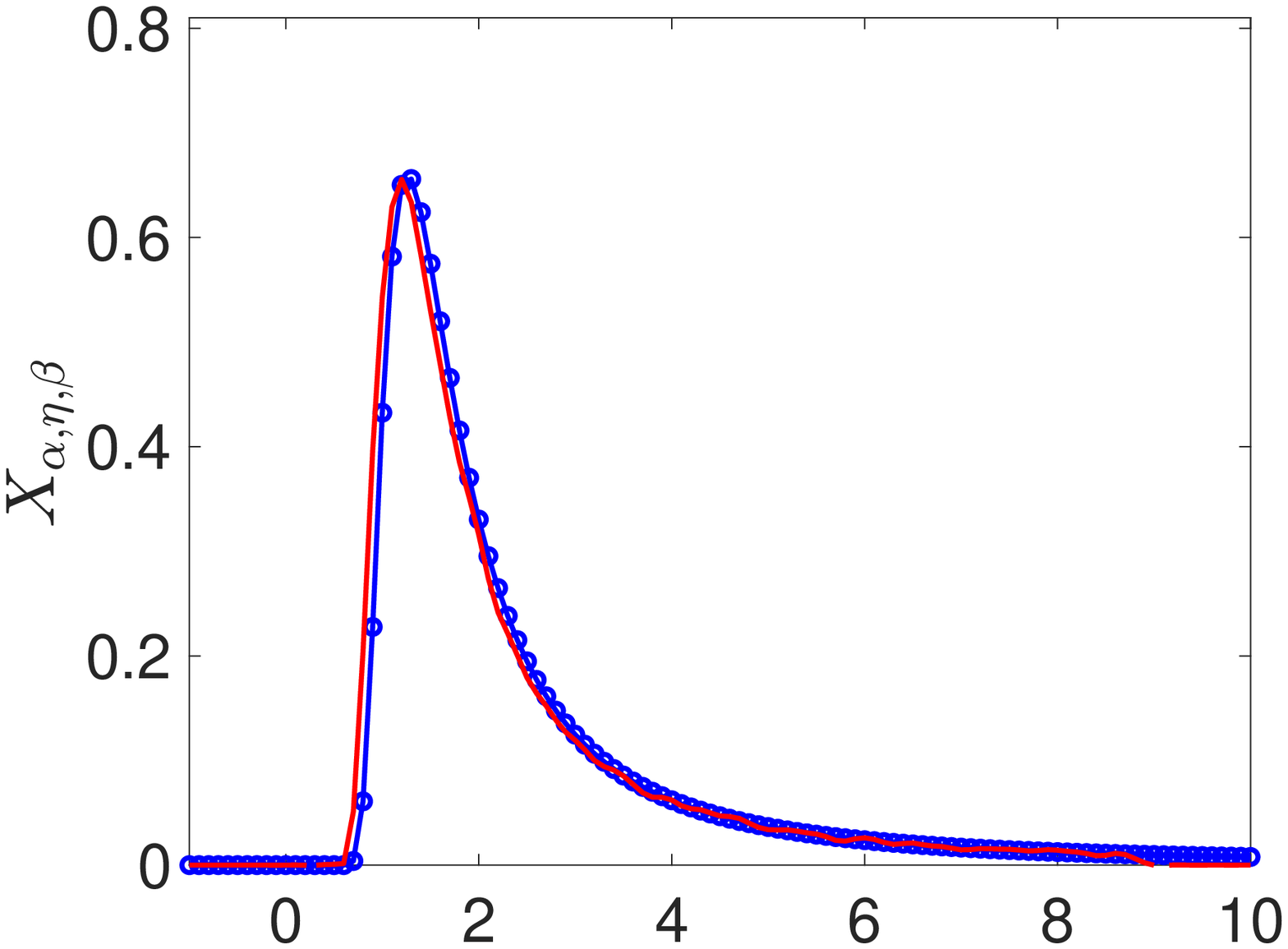}\\		
	\vspace{1mm}
	\includegraphics[width=19pc]{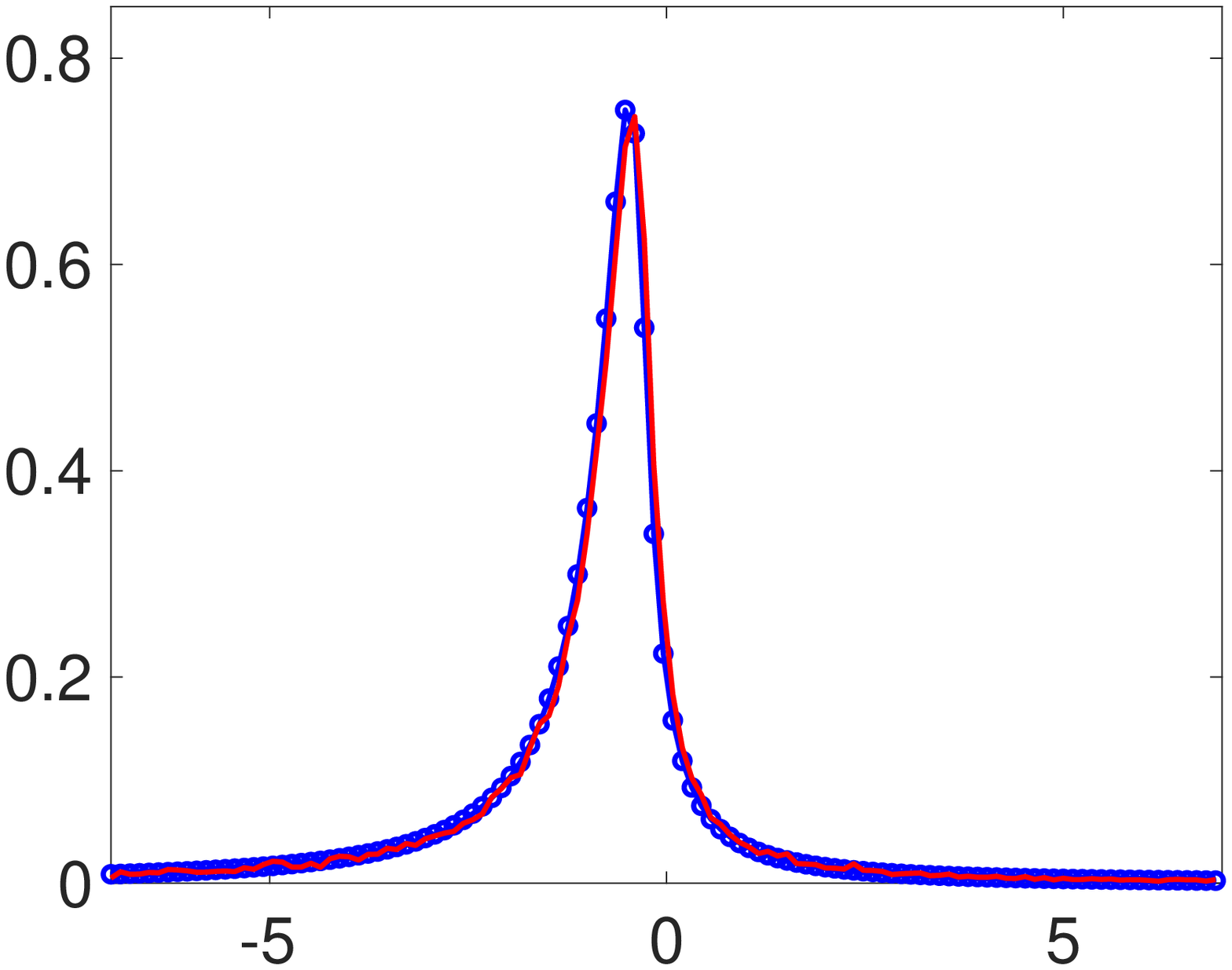}	 	
	\caption{Probability density functions for $\alpha$-stable laws $X_{\alpha,\eta,\beta}$ with $\alpha=0.8$, $\eta=0.5$ and (top): $\beta=0$, (middle): $\beta=1$ and (bottom): $\beta=-0.4$. The blue curve (open circles) uses the function \texttt{stblpdf} from the software package STABLE \cite{NolanStable}; the red continuous line shows the splined empirical histogram of the deterministic induced dynamics.}
	\label{fig:Lawpdf2}
\end{figure}

\begin{rmk} 
We expect that rigorous error rates can be obtained in Theorem~\ref{thm:stable1} and~\ref{thm:stable2} and that these rates will be poorest as $\alpha$ approaches $1$ and $2$ from below.
Indeed,  it is well-known even for sums of i.i.d.\ random variables that convergence rates to an $\alpha$-stable law are slow for $\alpha\in(1,2)$ close to $2$ and $\alpha\in(0,1)$ close to~$1$.
Indicative upper bounds on rates of convergence (ignoring logarithmic factors) for the distribution functions~\cite{Cramer63,Hall81} are 
$O(n^{-(2\alpha^{-1}-1)})$ for $\alpha\in(1,2)$ and
$O(n^{-(\alpha^{-1}-1)})+O(n^{-1})$ for $\alpha\in(0,1)$ with improvements for $\alpha<1$ if $\beta=0$.
Similar estimates for $\alpha\in(0,1)$ in a deterministic setting that is almost the same as the one here can be found in~\cite{Terhesiu14}.
Further work would be required to estimate the implied ``big O'' constant.
We do not address these issues further here.
\end{rmk}

\section{Generating $\alpha$-stable L\'evy processes}
\label{sec:Levy}

Given an $\alpha$-stable law 
$X_{\alpha,\eta,\beta}$, we define
the corresponding $\alpha$-stable L\'evy process to be the
c\`adl\`ag process $W_{\alpha,\eta,\beta}\in D[0,\infty)$ 
with independent stationary increments such that
$W_{\alpha,\eta,\beta}(t)=_d t^{1/\alpha}X_{\alpha,\eta,\beta}$.

The next result shows how to generate $\alpha$-stable L\'evy processes
$W_{\alpha,\eta,\beta}$ with
$\alpha\in(0,1)\cup(1,2)$, $\eta>0$, $\beta\in[-1,1]$.
For the proof, see Section~\ref{sec:proofs}.

\begin{thm}  \label{thm:Levy}
Assume the setup of Theorem~\ref{thm:stable2}.  Define
\[
W_n(t)=n^{-\gamma} d_\alpha^{-\gamma}\SMALL (\sum_{j=0}^{\lfloor nt\rfloor -1}\delta_j\tau_j-nt\beta \ell_\alpha),\quad t\ge0.
\]
Then $W_n$ converges weakly to $W_{\alpha,1,\beta}$ in
$D[0,\infty)$ as $n\to\infty$.

By Remark~\ref{rmk:eta} we can obtain all processes $W_{\alpha,\eta,\beta}=\eta W_{\alpha,1,\beta}$ in this way.    

In particular, taking $\delta_j\equiv\pm1$, we obtain processes $W_{\alpha,1,\pm 1}$ corresponding to the one-sided stable laws in 
Theorem~\ref{thm:stable1}. 
\end{thm}
As in Section~\ref{sec:stablegen}, weak convergence is understood with respect to the probability $\mu_Y$.
Convergence holds in the Skorohod $\cM_1$ topology on $D[0,\infty)$~\cite{Skorohod56,Whitt}.

Figure~\ref{fig:Levy} shows sample trajectories of L\'evy processes for $\alpha=1.6$, $\eta=0.5$ and various values of $\beta$ using the induced deterministic dynamics.

\begin{figure}
	\centering
	\includegraphics[width=19pc]{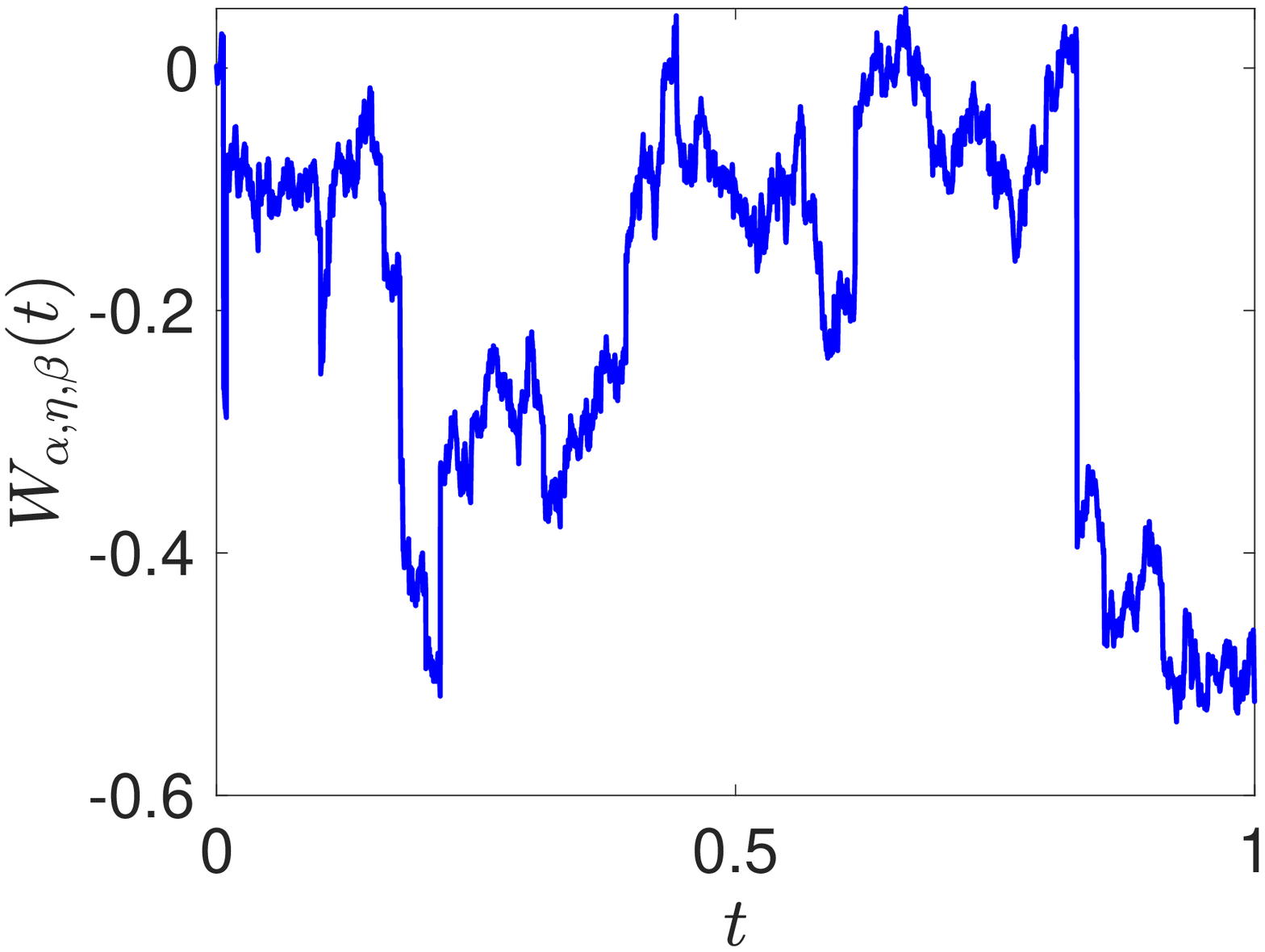}\\		
	\vspace{1mm}
	\includegraphics[width=19pc]{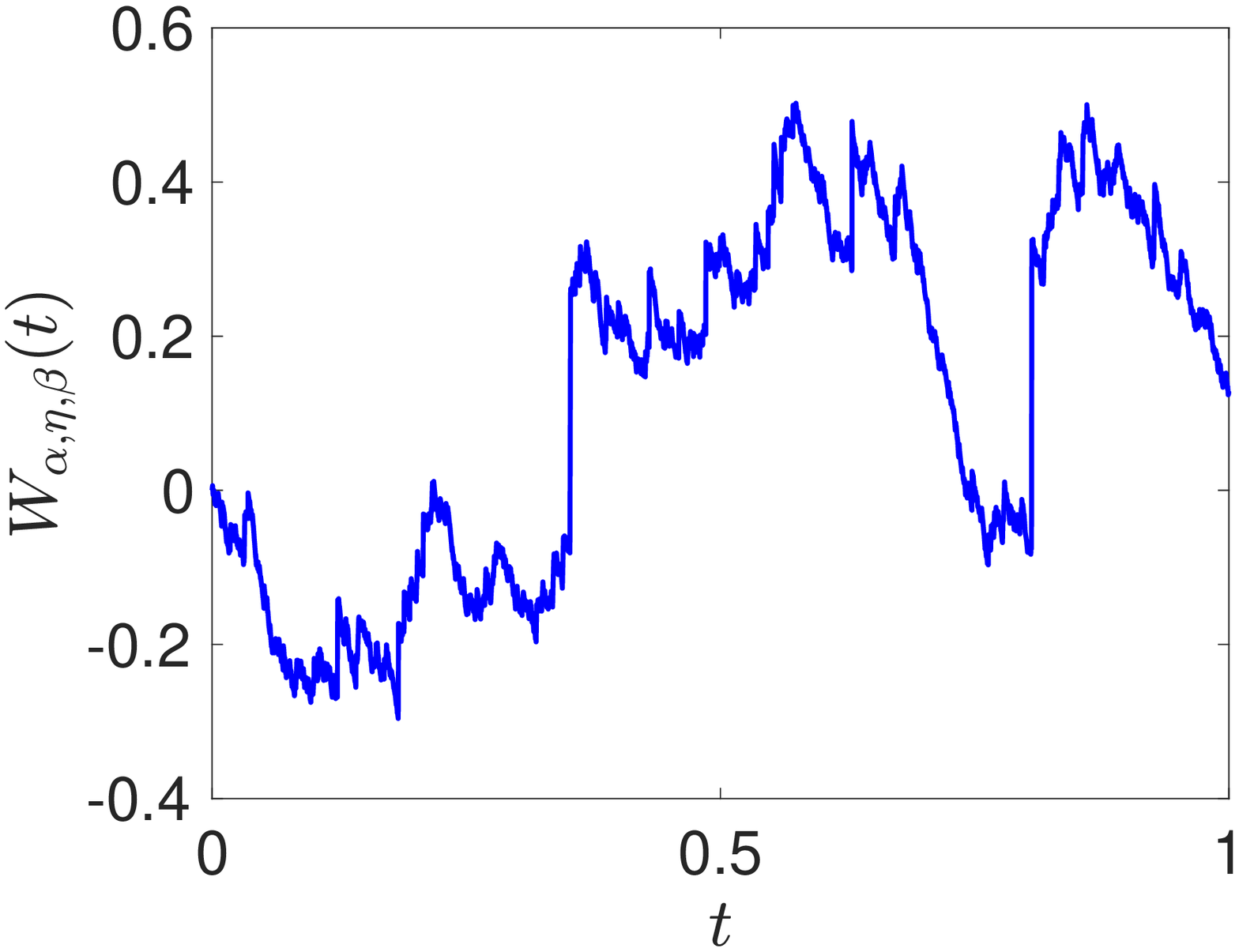}\\		
	\vspace{1mm}
	\includegraphics[width=19pc]{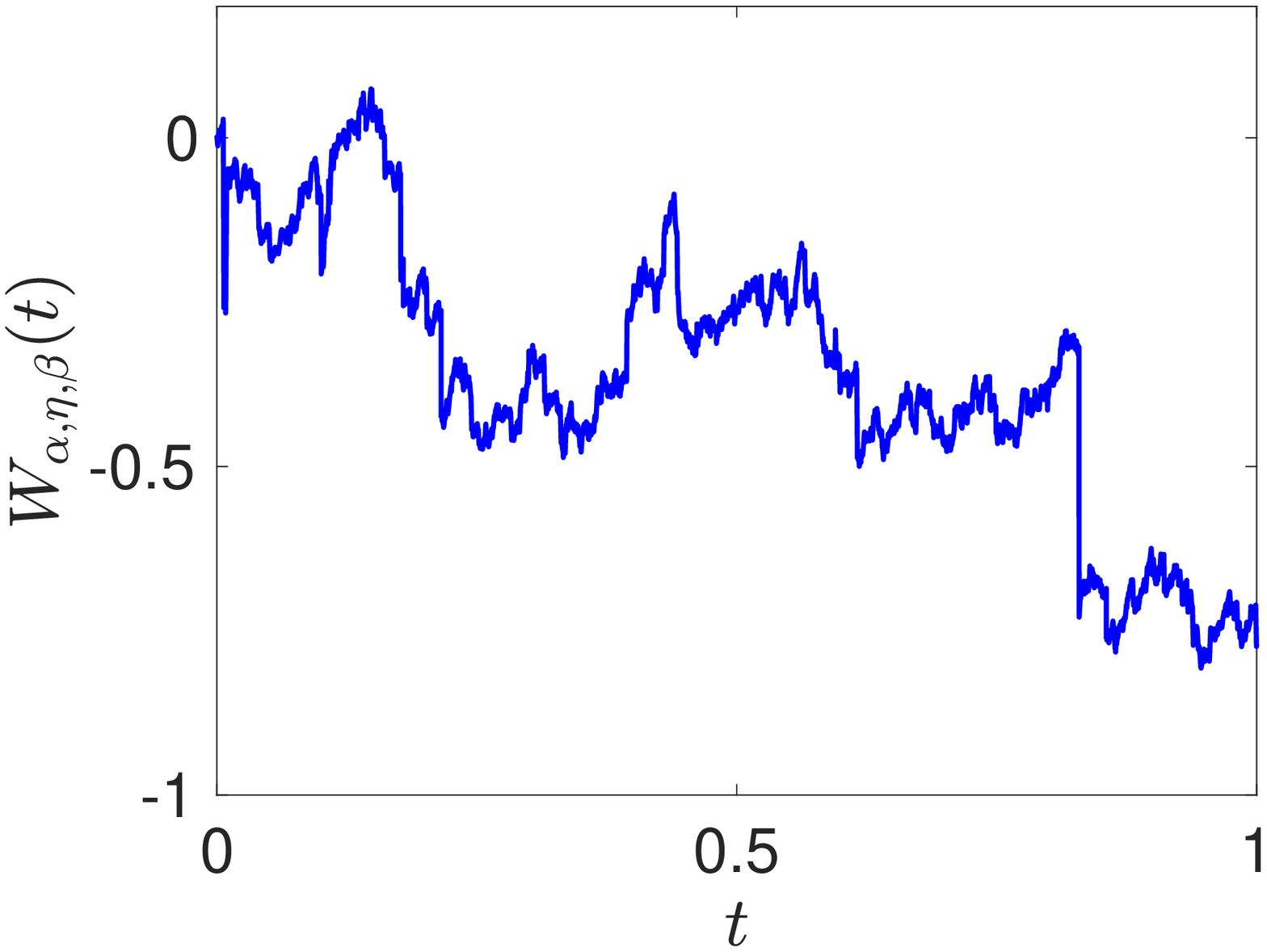}		
	\caption{Sample paths of  L\'evy processes $W_{\alpha,\eta,\beta}$ with $\alpha=1.6$, $\eta=0.5$ and (top): $\beta=0$, (middle): $\beta=1$ and (bottom): $\beta=-0.4$.}
	\label{fig:Levy}
\end{figure}

%%%%%%%%%%%%%%%%%%%%%%%%%%%%%%%%%%%%%%%%%%%%%%%%%%

\section{Numerical integration of SDEs using homogenisation}
\label{sec:SDE}

In this section we show how to simulate Marcus SDEs of the form (\ref{eq:SDE0}) with non-Lipschitz drift and diffusion terms driven by multiplicative L\'evy noise. 

The case of ``exact'' multiplicative noise where $m=d$ and $b=(Dg)^{-1}$ for some suitable function $g:\R^d\to\R^d$ was studied in~\cite{GottwaldMelbourne13c}.  In this case, the change of coordinates $\tZ=g(Z)$ leads to an SDE in terms of $\tZ$ with constant diffusion term.  In principle, $\tZ$ can now be computed by existing methods~\cite{HighamEtAl02,MilsteinTretyakov05,HutzenthalerEtAl12,Sabanis13,TretyakovZhang13,DareiotisEtAl16,Mao16,KumarSabanis17,KellyLord16} and then $Z$ is recovered via the formula $Z=g^{-1}(\tZ)$.

For $d\ge2$, exactness is a very restrictive condition.  Even for $d=1$, the method above is not useful when $b$ vanishes as in the examples below.  Hence our aim is to devise a numerical method that does not rely on exactness.

Our method in this section uses the full Thaler map $T:[0,1]\to[0,1]$ for which the density $h$ in~\eqref{eq:h} defines a finite measure only for $\gamma<1$.
Theorem~\ref{thm:SDE} below does not hold in the infinite measure setting and hence fails for $\gamma>1$.
Hence in this section we restrict to the range $\alpha\in(1,2)$.
(In contrast, our methods in Sections~\ref{sec:stable} and~\ref{sec:Levy} involve returns to the set $Y=[x^*,1]$ on which $h$ restricts to a finite measure for all $\gamma\in[0,1)\cup(1,\infty)$.)

Throughout this section we work with the invariant probability measure $\mu$ corresponding to the normalised density
\begin{equation}
 \label{eq:h2}
 \tilde h(x)=\frac{1-\gamma}{2^{1-\gamma}}(x^{-\gamma}+(x+1)^{-\gamma}).
\end{equation}

\subsection{Numerical algorithm for solving SDEs}

In this paper, we focus on solving SDEs of the type~\eqref{eq:SDE0} in the scalar case $d=m=1$.  The theoretical basis~\cite{ChevyrevEtAl19} behind the method applies in general dimensions.  However, in practice one would need to consider Thaler-type maps with multiple fixed points and to construct higher-dimensional processes $W_n\in D([0,\infty),\R^m)$ converging to the appropriate driving L\'evy process as in Section~\ref{sec:Levy}.
Since these preliminary steps have been carried out so far only in the scalar case, we restrict to that case here.

Consider the SDE~\eqref{eq:SDE0} with $d=m=1$ and $W=W_{\alpha,\eta,\beta}$
where $\alpha\in(1,2)$, $\eta>0$, $\beta\in[-1,1]$. 
Let $T$ be the Thaler map~\eqref{eq:Thaler} with $\gamma=\alpha^{-1}$.
We define a sequence of observables $v^{(n)}=\chi^{(n)}\,v\circ T^n$ where $v:[0,1]\to\R$ is the  mean zero observable given by
\[
v(x)= \eta d_\alpha^{-\gamma}(1-2^{\gamma-1})^{-\gamma} \tilde v(x), \qquad
\tilde v(x)=
\begin{cases} 
\phantom{Yy} 1 & x\le x^\star
\\ 
(1-2^{1-\gamma})^{-1} 
& x>x^\star
\end{cases},
\]
and
\[
\chi^{(n)}=\chi_{n-1}\cdots\chi_0\in\{\pm1\},
\qquad 
\chi_j=\begin{cases} 
\,1 & T^jx\le x^\star
\\ 
\BIG \delta_j
& T^jx>x^\star
\end{cases}.
\]
Here, $d_\alpha$ is as in \eqref{eq:da} and $\delta_0,\delta_1,\dots$ are independent copies of the random variable $\delta$ where $\PP(\delta=\pm1)=\frac12(1\pm\beta)$ as in Section~\ref{sec:stablegen}.
(In particular, the random variable $\chi^{(n)}$ gets updated only when the trajectory visits $Y$ and is unchanged during the laminar phase in $[0,x^*]$).

We can now state our main result (see Section~\ref{sec:proofs} for the proof).
\begin{thm}  \label{thm:SDE}
Let $a:\R\to\R$ be $C^{1+\delta}$ and $b:\R\to\R$ be $C^{\alpha+\delta}$ for some $\delta>0$.
Fix $\xi\in\R$.
Let
\begin{align} \label{eq:fs}
z_{n+1}^{(\eps)}=z_n^{(\eps)}+\eps a(z_n^{(\eps)})+\eps^\gamma b(z_n^{(\eps)})v^{(n)}, \quad z_0^{(\eps)}=\xi,
\end{align}
where $v^{(n)}:[0,1]\to\R$ is as defined above,
and set $\hat z_\eps(t)=z^{(\eps)}_{\lfloor t\eps^{-1}\rfloor }$.
Then $\hat z_\eps$ converges weakly to $Z$ in $D[0,\infty)$ on the probability space
$([0,1],\mu)$ as $\eps\to0$ where
$Z$ is the solution to the Marcus SDE~\eqref{eq:SDE0} with $Z(0)=\xi$.
\end{thm}

\begin{rmk}
We refer to equation~\eqref{eq:fs} as a fast-slow map.
Indeed, in the case $\beta=1$ ($\delta_n\equiv1$) Theorem~\ref{thm:SDE} ensures that solutions $z_n^{(\eps)}$ of the fast-slow system
\begin{align*}
z_{n+1}^{(\eps)} & =z_n^{(\eps)}+\eps a(z_n^{(\eps)})+\eps^\gamma b(z_n^{(\eps)})v(x_n), \quad z_0^{(\eps)}=\xi, \\
x_{n+1} & =Tx_n
\end{align*}
converge weakly to solutions of the SDE~\eqref{eq:SDE0} on the slow time scale, i.e.\ $\hat z_{\eps}=z^{(\eps)}_{\lfloor \cdot\,  \eps^{-1} \rfloor}\to_w Z$ as $\eps\to 0$.
In the general case $\beta\in[-1,1]$, there is a similar but more complicated interpretation that is used in the proof of Theorem~\ref{thm:SDE} in Section~\ref{sec:proof}.
\end{rmk}

\begin{rmk}  
The topology used for the weak convergence in Theorem~\ref{thm:SDE} is too technical to define here and we refer to \cite{ChevyrevEtAl19}.   It is weaker than the $\cM_1$ topology, but sufficiently strong to guarantee convergence in the sense of joint distributions.
That is, $(\hat z_\eps(t_1),\dots,\hat z_\eps(t_k))$ converges in distribution
to $(Z(t_1),\dots,Z(t_k))$ in $\R^k$ as $\eps\to0$ for all $t_1,\dots,t_k\in[0,1]$, $k\ge1$.
\end{rmk}

\begin{rmk}  \label{rmk:leb}
By results of~\cite{Eagleson76,Zweimuller07} (see in particular~\cite[Example~1.1]{CFKMZsub}), the initial conditions $x_0\in[0,1]$ can be equally well (from the theoretical point of view of Theorem~\ref{thm:SDE}) chosen using the invariant probability measure $\mu$ or the uniform Lebesgue measure.
We have checked numerically in the case $a\equiv0$, $b\equiv1$ (corresponding to generation of a L\'evy process $Z = W_{\alpha,\beta,\eta}$) that convergence of the probability density at $t=1$ is faster if the initial conditions are drawn using $\mu$.

Hence throughout this section, when applying the fast-slow map~\eqref{eq:fs}, we work with initial conditions $x_0$ drawn using the invariant probability measure $\mu$.  The explicit formula for the density $\tilde h$ in~\eqref{eq:h2} is less helpful here due to the singular behaviour near $x=0$.  To circumvent this, we propagate
uniformly distributed initial conditions $x_0'\in[0,1]$  under $10,000$ iterations of the Thaler map and then work with the initial conditions $x_0=T^{10,000}x_0'$.
\end{rmk}

\subsection{Numerical results for solving SDEs}
\label{sec:numerics}
To illustrate our method, we consider the dynamics of a particle in a double-well potential $V$ driven by a L\'evy process
\begin{align}
dZ=-\nabla V(Z)\, dt + b(Z)\diamond dW_{\alpha,\eta,\beta} 
\label{eq:SDE_LMW}
\end{align}
with  drift term $a= -\nabla V$. We consider two specific examples with non-Lipschitz drift and diffusion terms. 
In the first example, our approach is in good agreement with conventional methods.  The second example possesses a natural boundary which seems better treated by the deterministic method presented in this paper. \\
%The first one allows for a simple reformulation into an SDE with additive noise whereas the second example cannot be easily cast into a form amenable for standard numerical techniques. \\

{\em{Example 1}}: Consider the SDE~\eqref{eq:SDE_LMW} with potential and diffusion terms
\[
V(Z)=A[(Z-a_0)^2/b_0^2-1]^2 \qquad\text{and} \qquad 
b(Z)= s \sqrt{1-(Z/B)^2}.
\]
 This example was considered in \cite{LiEtAl13} where the stochastic forcing was a compound Poisson process. Note that both the drift and diffusion terms are non-Lipschitz. We use the parameters $A=20$, $a_0=400$, $b_0=2$, $B=500$ from~\cite{LiEtAl13}, and take $s=10$ for the strength of the diffusion. 
We take $\alpha=1.5$, $\eta=0.5$, $\beta=0$  for the driving L\'evy process
$W_{\alpha,\eta,\beta}$.

Theorem~\ref{thm:SDE} implies in particular convergence in distribution of $\hat z_\eps(t)$ to the stochastic process $Z(t)$ at fixed $t$.
We test this numerically by generating the probability density function of $Z(1)$ via (i) existing methods based on Euler-Maruyama discretisation and (ii) our theorem.  The results are shown in Figure~\ref{fig:SDE2_t1}. \\

First we describe method~(i).
The non-Lipschitz diffusive term $b$ can be removed by the change of coordinates $\tZ=g(Z)=B\arcsin\tfrac{Z}{B}$.  The transformed SDE is
\begin{align}
d\tZ=\tilde a(\tZ)\,dt + s\, dW_{\alpha,\eta,\beta},
\label{eq:SDE_LMWt}
\end{align}
where the transformed drift term 
\[
\tilde a(\tZ)=-\frac{4A}{b_0^4}\frac{1}{|\cos\frac{\tZ}{B}|}(B\sin\tfrac{\tZ}{B}-a_0)((B\sin\tfrac{\tZ}{B}-a_0)^2-b_0^2)
\]
now has a singularity at $\tZ=\pm \frac{\pi}{2}B$ corresponding to $Z=\pm B$. For the parameter values above, it turns out that the singularity lies outside the range where the probability density function is significantly different from zero and is relatively harmless. 
The transformed SDE (\ref{eq:SDE_LMWt}) for $\tZ$ can now be solved with an Euler-Maruyama type
scheme with time step $\Delta t$. To account for the non-Lipschitz drift term $\tilde a$, we apply the taming method \cite{HutzenthalerEtAl12,Sabanis13}, and discretise according to
\[
\tZ_{n+1}=\tZ_n + \frac{\tilde a(\tZ_n)}{1+|\tilde a(\tZ_n)|\Delta t}    \Delta t + s\, \Delta W_{\alpha,\eta,\beta},
\]
where $\Delta W_{\alpha,\eta,\beta} =_d (\Delta t)^\gamma X_{\alpha,\eta,\beta}$.  Finally, we transform back to recover the solution $Z=g^{-1}(\tZ)=B\sin\tfrac{\tZ}{B}$ to the original SDE~\eqref{eq:SDE_LMW}. 
In Figure~\ref{fig:SDE2_t1}, we applied the Euler-Maruyama method with
time step $\Delta t=0.0001$ averaged over $500,000$ realisations of the driving L\'evy noise, starting from an initial condition $Z(0)=\xi=410$. 

Method~(ii) consists of applying 
Theorem~\ref{thm:SDE} directly to the non-transformed SDE.
Figure~\ref{fig:SDE2_t1} shows the empirical distribution of $\hat z_\eps(1)$  averaged again over $500,000$ realisations $x_0=T^{10,000}x_0'$ (as explained in Remark~\ref{rmk:leb}) for various values of $\epsilon$ with initial condition $Z(0)=z_0^{(\eps)}=\xi=410$. The convergence of the probability density function obtained by iterating the fast-slow map (\ref{eq:fs}) and Theorem~\ref{thm:SDE} is clearly seen. \\ 

{\em{Example 2}}: 
Consider now the SDE (\ref{eq:SDE_LMW}) with potential and diffusion terms
\[
V(Z)=\tfrac{1}{2}Z^2-\tfrac{1}{4}Z^4 \qquad\text{and} \qquad 
b(Z)= -Z^2.
\]
We take $\alpha=1.5$, $\eta=0.5$, $\beta=0.5$  for the driving L\'evy process $W_{\alpha,\eta,\beta}$. There is a natural boundary at $Z=0$: for $Z(0)>0$ the stochastic process remains strictly positive for all times with probability $1$. This is readily seen by writing the SDE as $dZ=Zg_1(Z)\,dt+Zg_2(Z)\diamond dW$ where $g_1(Z)=1-Z^2$ and $g_2(Z)=-Z$.  Since the Marcus integral satisfies the standard laws of calculus, solutions $Z(t)$ satisfy $Z(t)=Z(0)\exp\{\int_0^t g_1(Z(s))\,ds+\int_0^t g_2(Z(s))\diamond dW(s)\}$. Hence the sign of the initial condition is preserved. 

Again, we compare the two methods~(i) Euler-Maruyama and~(ii) Theorem~\ref{thm:SDE}.   As shown below, Euler-Maruyama fails to deal adequately with the natural boundary at $Z=0$, whereas Theorem~\ref{thm:SDE} respects this boundary.\\

To apply Euler-Maruyama, we again start by removing the non-Lipschitz diffusion term via 
the change of coordinates $\tZ=g(Z)=Z^{-1}$. The transformed SDE is
\begin{align}
d\tZ=(\tZ^{-1}-\tZ)\,dt + dW_{\alpha,\eta,\beta}.
\label{eq:SDE1_trans}
\end{align}
When discretising the transformed SDE (\ref{eq:SDE1_trans}) using an Euler-Maruyama scheme, however, large increments $\Delta W_{\alpha,\beta,\eta}$ lead to spurious crossings of the natural boundary at $Z=0$. This does not occur for our deterministic method applying Theorem~\ref{thm:SDE} directly to the non-transformed SDE. We show in Figure~\ref{fig:SDE_t1} the probability density function of $Z(2)$ obtained  by considering the empirical distribution of $\hat z_\eps(2)$ for several values of $\eps$. We compute the latter by averaging over $500,000$ realisations for various values of $\epsilon$ with initial condition $Z(0)=z_0^{(\eps)}=\xi=0.2341$. The corresponding probability density function for an Euler-Maruyama discretisation with time step $\Delta t=0.0001$ is shown as well. Whereas the empirical density obtained from the fast-slow map converges to a unimodal probability density function, the probability density function obtained from the Euler-Maruyama discretisation exhibits significant leakage into the region $Z<0$.  \\

\begin{figure}
	\centering
	\includegraphics[width=25pc]{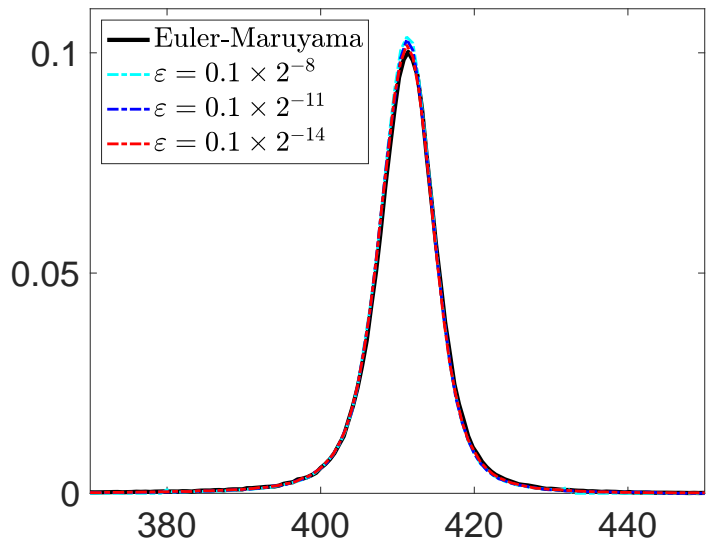}
	\caption{Probability density function for the solution to the SDE in Example~1 at fixed time $t=1$.  Results for the fast-slow map (\ref{eq:fs}) are shown for several values of~$\eps$ and are compared with Euler-Maruyama discretisation.}
	\label{fig:SDE2_t1}
\end{figure}

\begin{figure}
        \centering
        \includegraphics[width=25pc]{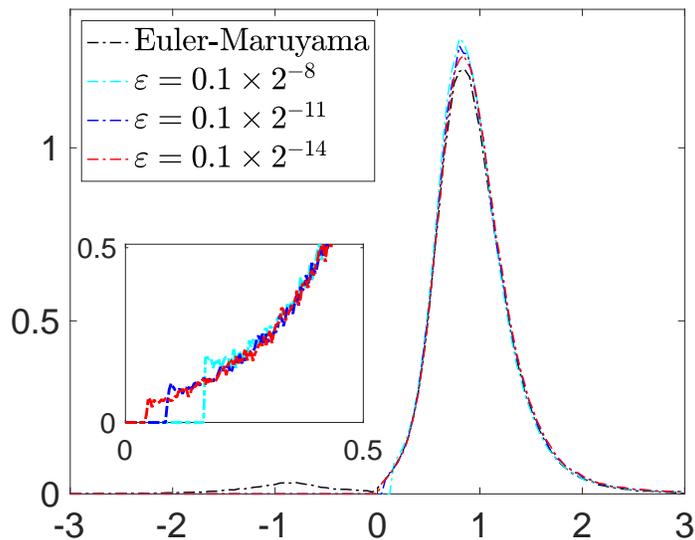}
        \caption{Probability density function for the solution to the SDE in Example~2 at fixed time $t=2$.  Results for the fast-slow map (\ref{eq:fs}) are shown for several values of~$\eps$  and are compared with Euler-Maruyama discretisation. The inset shows a zoom near the natural boundary at $Z=0$ for the probability density function obtained from the fast-slow map (\ref{eq:fs}).}
        \label{fig:SDE_t1}
\end{figure}

We end with a few comments on numerical issues when iterating the fast-slow map (\ref{eq:fs}). The smallness of $\epsilon$ requires long simulations as the convergence is on the slow time scale $n=\lfloor\eps^{-1}t\rfloor$.  As a result, the fast dynamics may get trapped on a spurious periodic orbit, caused by the discreteness of floating numbers. To avoid this, we occasionally add a normally distributed random number with mean zero and variance $10^{-20}$ (computed $\bmod 1$). This perturbation is added each time the fast orbit $x_n$ enters the hyperbolic region $[x^*,1]$ and has undergone at least $10^4$ iterations after the previous perturbation -- this ensures that the superdiffusive statistics are not altered by the addition of the small perturbation.

\subsection{Numerical results on the stationary density and the auto-correlation function}
\label{sec:numerics2}
Moving beyond the theoretical justification provided by Theorem~\ref{thm:SDE}, in this subsection  we show that our method is furthermore able to provide a good approximation for the stationary density as estimated from large $t$  simulations as well as capturing temporal statistics. 

Figure~\ref{fig:SDE2} shows the stationary density for the SDE in
Example~1.
Again we compare (i)  Euler-Maruyama discretisation and (ii) Theorem~\ref{thm:SDE}.
For Euler-Maruyama, we take
 $\Delta t=0.001$ and generate a time series which is sampled every $2$ time units for a total of $t=2\times 10^6$ time units. The results  from the deterministic fast-slow map~\eqref{eq:fs} are shown to converge as $\eps$ decreases although there are spurious narrow peaks to the left and right of the large peaks associated with the minima of the potential $V$. 
The spurious peaks decrease in size and move further away from the relevant part of the stationary measure as $\eps$ decreases. 
They are caused by unstable fixed points $z^\star$ of the fast-slow map~\eqref{eq:fs} which converge to $Z=\pm B$ as $\eps\to0$. \\

Figure~\ref{fig:SDE} shows the stationary density in Example~2 obtained from using the fast-slow map (\ref{eq:fs}) for large $t$ for several values of $\eps$. The plots were generated to reach to times $t=5\times 10^7$ time units, sampled every $100\eps^{-1}$ steps. We show the relevant part of the stationary density as well as the tails at $0$ and $\infty$. We again observe spurious narrow peaks in the tails caused by the fixed points of the slow map with $v\equiv
\pm \eta d_\alpha^{-\gamma}(1-2^{\gamma-1})^{-\gamma}$
which are the values of $v$ on $[0,x^*)$ where the fast dynamics spends most of its time.  These fixed points are given by 
\[
z^\star=0, \qquad 
z^\star = -p\pm\sqrt{p^2+1}
\]
with $p=\pm\frac12\eps^{\gamma-1}d_a^{-\gamma}(1-2^{\gamma-1})^{-\gamma}$. Hence $z^\star\to0, \, \pm\infty$ as $\eps\to 0$. We remark that the Euler-Maruyama discretisation leads to a bimodal stationary density, rather than to a unimodal stationary density with support $(0,\infty)$.\\ 

\begin{figure}
	\centering
	\includegraphics[width=27pc]{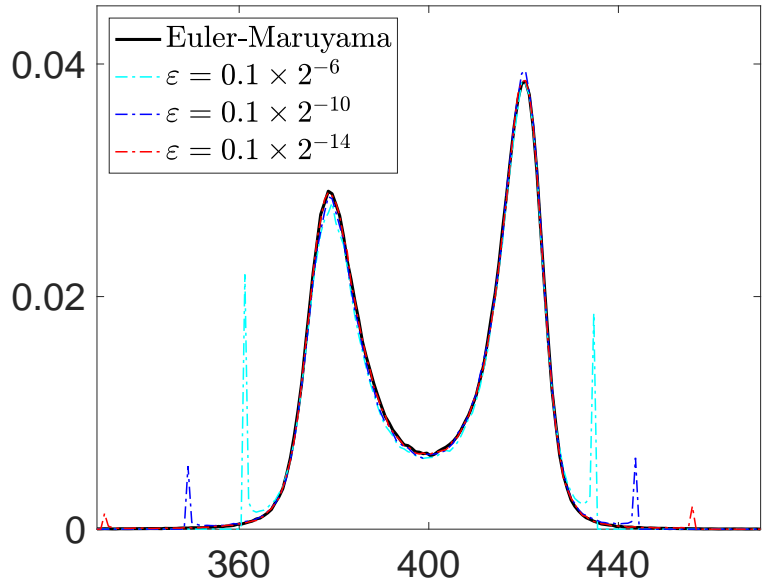}
	\caption{
Stationary density for the SDE in Example~1.  Results for the fast-slow map (\ref{eq:fs}) are shown for several values of $\eps$ and are compared with Euler-Maruyama discretisation.}
	\label{fig:SDE2}
\end{figure}

\begin{figure}
	\centering
	\includegraphics[width=17pc]{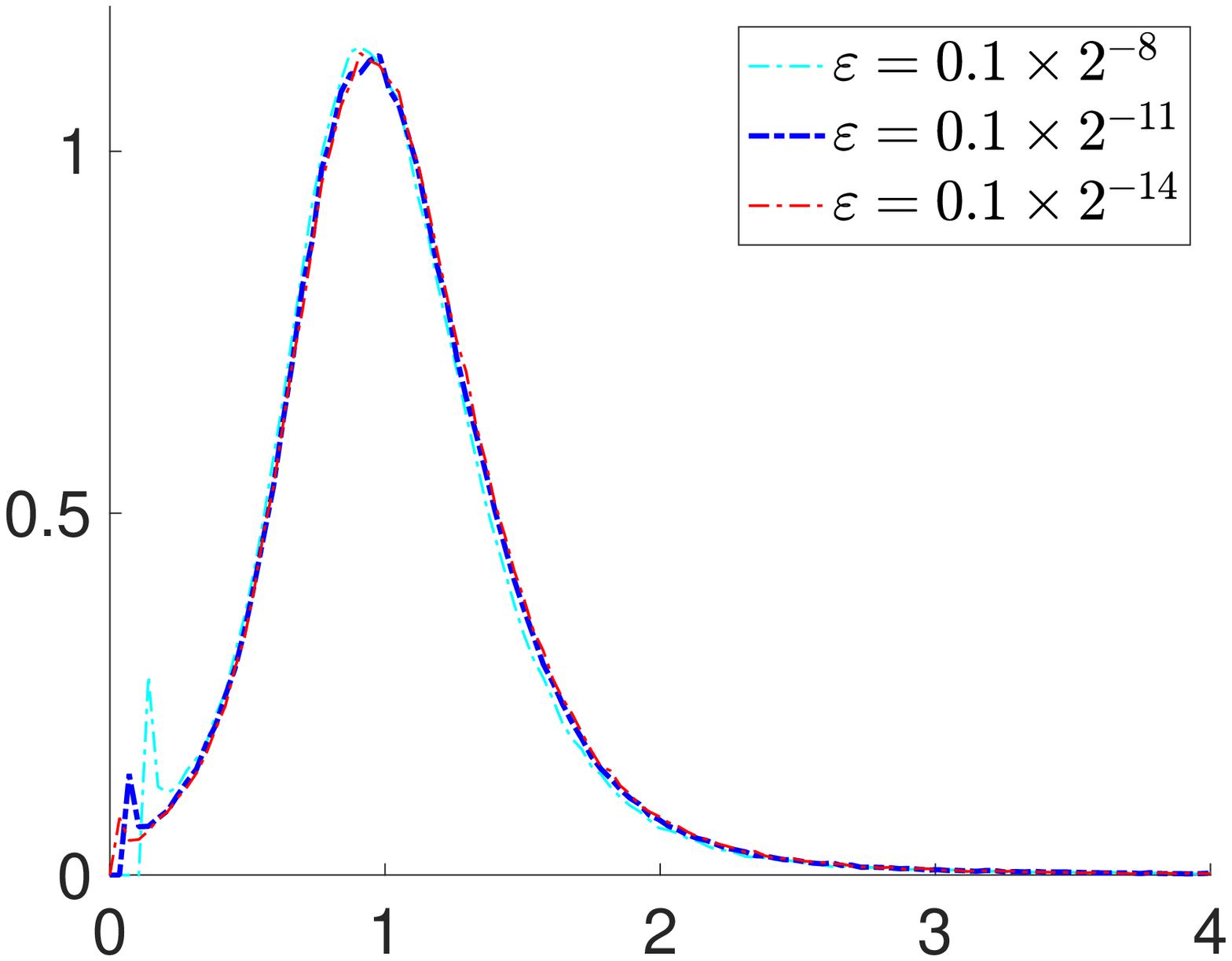}\;    
%       \vspace{1mm}
        \includegraphics[width=17pc]{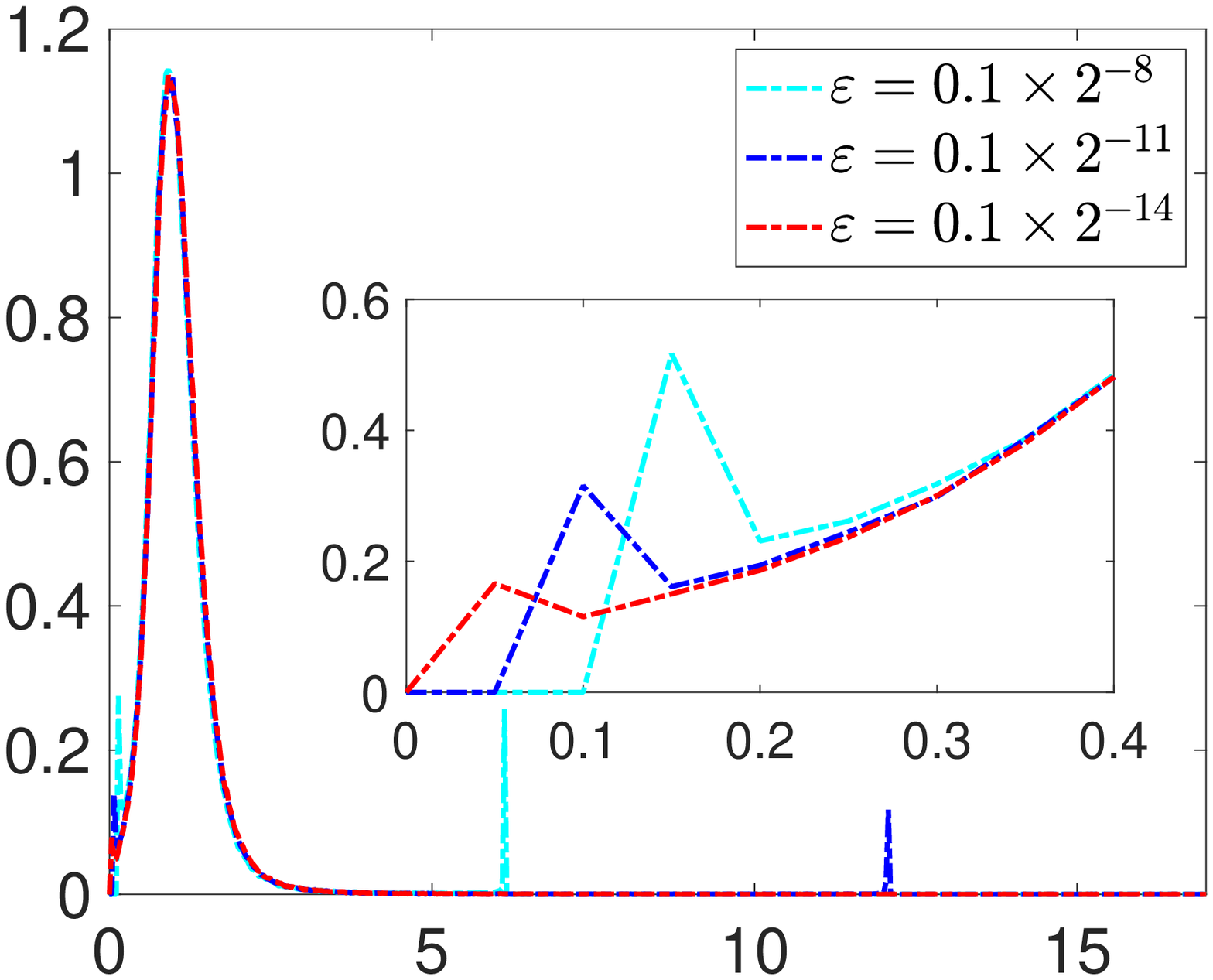}
 \caption{Stationary density for the SDE in Example~2.  Results for the fast-slow map (\ref{eq:fs}) are shown for several values of $\eps$.
Left: Relevant range. Right: Close-up of the spurious peaks for the stationary density computed using (\ref{eq:fs}).}
% Bottom: corresponding driving L\'evy process $W_{\alpha,\eta,\beta}$.}
	\label{fig:SDE}
\end{figure}

Moreover, our method is able to resolve temporal statistics of the underlying SDE. In Figure~\ref{fig:SDE2_C}, we compute the normalised auto-correlation function
\[
C(t) = \frac{1}{{\rm{Var}}[Z]}\int_0^\infty (Z(t+s)-\bar Z)(Z(s)-\bar Z)\, ds
%C(n) = \frac{c_n}{c_0} \qquad {\rm{with}} \qquad c_n=\frac{1}{N}\sum_{j=0}^{N-n} (z_{n+j}-\bar z)(z_j-\bar z)e
\]
%sample auto-correlation function
of solutions $Z$ to the SDE in Example~1
using the fast-slow map~\eqref{eq:fs} for various values of $\eps$. It is seen that the auto-correlation function converges to the reference auto-correlation function estimated from the time series obtained using the Euler-Maruyama method. The auto-correlation function is estimated using the same data used to obtain Figure~\ref{fig:SDE2}. 
We remark that whereas a time step of $\Delta t=0.001$ was sufficient to obtain the stationary density shown in Figure~\ref{fig:SDE2} using the Euler-Maruyama discretisation, the estimation of the auto-correlation function requires a smaller time step of $\Delta t=0.0001$, making Euler-Maruyama schemes more costly if resolving temporal statistics is required.\\

\begin{figure}
	\centering
	\includegraphics[width=27pc]{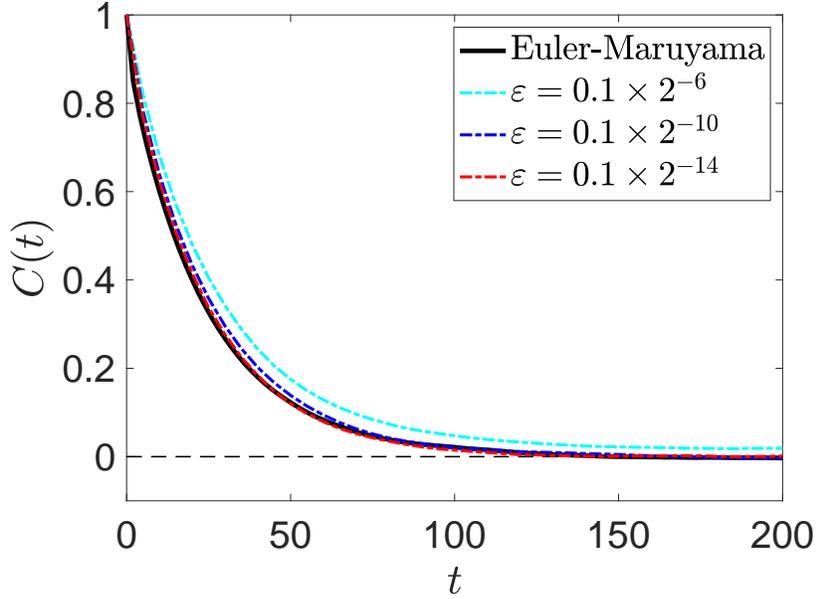}
	\caption{Auto-correlation function $C(t)$ of solutions $Z$ for the SDE in Example~1 estimated from the fast-slow map (\ref{eq:fs}) for several values of $\eps$, and from a direct discretisation using the Euler-Maruyama method as a reference. }
	\label{fig:SDE2_C}
\end{figure}

%%%%%%%%%%%%%%%%%%%%%%%%%%%%%%%%%%%%%%%%%%%%%%%%%%

\section{Proof of convergence of the algorithms}
\label{sec:proofs}

In this section we prove Theorems~\ref{thm:stable1},~\ref{thm:stable2},~\ref{thm:Levy} and~\ref{thm:SDE}.
%%%%%%%%%%%%%%%%%%%%%%%%%%%%%%%%%%%%%%%%%%%%%%%%%%

\subsection{Background on Gibbs-Markov maps}
\label{sec:GM}

We begin by defining the notion of Gibbs-Markov map following~\cite{Aaronson,AaronsonDenker01,ADU93}.
Suppose that $(Y,\mu_Y)$ is a probability space with
an at most countable measurable partition $\{Y_j,\,j\ge1\}$
and let $F:Y\to Y$ be a measure-preserving transformation.
We say that $F$ is full-branch if 
$F|_{Y_j}:Y_j\to Y$ is a measurable bijection for each $j\ge1$.

For $y,y'\in Y$, define the {\em separation time} $s(y,y')$ to be the least integer $n\ge0$ such that $F^ny$ and $F^ny'$ lie in distinct partition elements in $\{Y_j\}$.
It is assumed that the partition $\{Y_j\}$ separates trajectories, so $s(y,y')=\infty$ if and only if $y=y'$.  

\begin{defn} \label{def:GM}
A full-branch measure-preserving transformation map $F:Y\to Y$ is called a {\em Gibbs-Markov map} if it satisfies the following {\em bounded distortion} condition:
There exist constants $C>0$, $\theta\in(0,1)$ such that
the potential function $p=\log\frac{d\mu_Y}{d\mu_Y\circ F}:Y\to\R$ satisfies
\[
|p(y)-p(y')|\le C\theta^{s(y,y')} 
\]
for all $y,y'\in Y_j$, $j\ge1$.
\end{defn}

An observable
$V:Y\to\R$ is {\em locally constant} if $V$ is constant on partition elements.

\begin{thm}[Aaronson \& Denker]  \label{thm:Igor}
Let $F:Y\to Y$ be a Gibbs-Markov map with probability measure $\mu_Y$ and let
$V:Y\to\R$ be a locally constant observable.
Suppose that 
\[
\mu_Y(V>x) = (c_1+o(1))x^{-\alpha}, \qquad  
\mu_Y(V<-x) = (c_2+o(1))x^{-\alpha}\quad\text{as $x\to\infty$}, 
\]
where 
$\alpha\in(0,1)\cup(1,2)$, $c_1,c_2\ge0$, $c_1+c_2>0$.

Then
\[
n^{-1/\alpha}\Big(\sum_{j=0}^{n-1}V\circ F^j\,-\,a_n\Big)\to_d X_{\alpha,\eta,\beta} \quad\text{as $n\to\infty$}
\]
on the probability space $(Y,\mu_Y)$, where
\[
	\eta=\big((c_1+c_2)g_\alpha\big)^{-1/\alpha}, \qquad\beta=\frac{c_1-c_2}{c_1+c_2},
\]
with $g_\alpha$ as in~\eqref{eq:ga}, and
\(
a_n=\begin{cases} \hphantom{Y}0 & \alpha\in(0,1) \\
n\int_Y V\,d\mu_Y & \alpha\in(1,2)  \end{cases}.
\)
\end{thm}

\begin{proof}
This is a special case of~\cite{AaronsonDenker01}.
\end{proof}

\begin{rmk} \label{rmk:Igor}
The constraints on $V$ in Theorem~\ref{thm:Igor} guarantee that $V$ lies in the domain of the stable law $X_{\alpha,\eta,\beta}$.  That is, if $Z_1,Z_2,\ldots$ are i.i.d.\ copies of $V$, then 
$n^{-1/\alpha}\big(\sum_{j=1}^nZ_j\,-\,a_n\big)\to_d X_{\alpha,\eta,\beta}$
as $n\to\infty$.
Theorem~\ref{thm:Igor} guarantees that this remains true even though the increments $V\circ F^j$ are not independent in general.
\end{rmk}

\begin{thm}[Tyran-Kami{\'n}ska] \label{thm:TK}
Assume the set up of Theorem~\ref{thm:Igor}
and define the sequence of c\`adl\`ag processes $W_n(t)=n^{-1/\alpha}(\sum_{j=0}^{[nt]-1}V\circ F^j- a_n t)$ on the probability space $(Y,\mu_Y)$.
Then
$W_n\to_w W_{\alpha,\eta,\beta}$ in the 
Skorokhod $\cJ_1$-topology on $D[0,\infty)$ as \mbox{$n\to\infty$}.  
\end{thm}

\begin{proof}
This is a special case of~\cite{TyranKaminska10}.
\end{proof}

\subsection{Induced Thaler maps}
\label{sec:F}

Let $T:[0,1]\to[0,1]$ be a Thaler map as defined in~\eqref{eq:Thaler} with parameter
$\gamma\in(0,1)\cup(1,\infty)$.
For each $\gamma$, there is a unique (up to scaling) $\sigma$-finite absolutely continuous invariant measure $\mu$ with density $h$ as in~\eqref{eq:h}, and
$\mu$ is finite if and only if $\gamma<1$.

Let $Y=(x^\star,1]$.  We consider the first return time 
$\tau:Y\to\Z^+$ and the first return map $F=T^\tau:Y\to Y$,
\[
\tau(y)=\inf\{n\ge1:T^ny\in Y\}, \qquad Fy=T^{\tau(y)}y.
\]
We refer to $F$ as the {\em induced Thaler map}.
The probability measure $\mu_Y=\mu|_Y/\mu(Y)$
is $F$-invariant and ergodic.

\begin{prop} \label{prop:tau} For each $\gamma\in(0,1)\cup(1,\infty)$, we have
	$\mu_Y(\tau>n)\sim e_\alpha n^{-\alpha}$ as $n\to\infty$ where $\alpha=\gamma^{-1}$ and
$\BIG e_\alpha= \alpha^{\alpha}\frac{1-\gamma}{2^{1-\gamma}-1}=d_\alpha g_\alpha^{-1}$.
\end{prop}

\begin{proof}
\textbf{Step 1:} 
Let $x_n$ be the decreasing sequence in $(0,x^\star]$, such that $Tx_{n+1}=x_n$, $n\ge1$.
Note that $Tx=x(1+x^\gamma+O(x^{2\gamma}))$ on $[0,x^\star]$.
Let $\phi:[0,1]\to[0,x^\star]$ be the inverse of this branch and write
\[
	\phi(x)=x(1-x^\gamma \psi(x)), \quad \psi(x)=1+O(x^\gamma).
\]
Then
\begin{align*}
	\phi(x) & = [x^{-\gamma}(1-x^\gamma \psi(x))^{-\gamma}]^{-1/\gamma}
 = [x^{-\gamma}+\gamma\hat \psi(x)]^{-1/\gamma},
\end{align*}  
where $\hat \psi(x)=1+O(x^\gamma)$.
Inductively,
\[
	\phi^nx=\Bigl[x^{-\gamma}+\gamma\sum_{j=0}^{n-1}\hat \psi(\phi^jx)\Bigr]^{-1/\gamma}.
\]
Set $x=1$ so $\phi^nx=x_n\to0$.
Then $\sum_{j=0}^{n-1}\hat \psi(\phi^jx)=\sum_{j=0}^{n-1}\hat \psi(x_j)=n+o(n)$ as $n\to\infty$.  Hence
\[
x_n=\phi^n1=[1+\gamma n+o(n)]^{-1/\gamma}
\sim (\gamma n)^{-1/\gamma}=\alpha^\alpha n^{-\alpha}.
\]

\noindent\textbf{Step 2:} 
Now let $y_n\in(x^\star,1]$ with $Ty_n=x_n$.
Let $T_2=T|_Y$ be the second branch and note that
$T_2$ maps the interval $[x^\star,y_n]$ onto $[0,x_n]$.
By the mean value theorem
\[
x_n-0=T_2'(y)(y_n-x^\star),
\]
for some $y\in [x^\star,y_n]$.
Moreover $|T_2'(y)-T_2'(x^\star)|\le |T_2''|_\infty(y-x^\star)\ll y_n-x^\star\to0$
as $n\to\infty$.
Hence $T_2'(y)\sim T_2'(x^\star)$.
Combining these calculations with step 1, we have
\[
y_n-x^\star
\sim (T'(x^\star))^{-1}x_n
\sim (T'(x^\star))^{-1}\alpha^{\alpha}n^{-\alpha}.
\]
Now,
\begin{align*}
T'(x^\star) & =({x^\star}^{1-\gamma}+(1+x^\star)^{1-\gamma}-1)^{\gamma/(1-\gamma)}
\{{x^\star}^{-\gamma}+(1+x^\star)^{-\gamma}\}
\\ & =
\{{x^\star}^{-\gamma}+(1+x^\star)^{-\gamma}\}
=h(x^*).
\end{align*}
Hence
$y_n-x^\star \sim \alpha^\alpha h(x^*)^{-1}n^{-\alpha}$.

\noindent\textbf{Step 3:} 
We use the formula for the density in~\eqref{eq:h}.
Observe that
\begin{align*}
\mu_Y(\tau>n) & =\mu(Y)^{-1}\int_{x^\star}^{y_n}h(y)dy
\\ & = \mu(Y)^{-1}(y_n-x^\star)h(x^\star)+\mu(Y)^{-1}\int_{x^\star}^{y_n}(h(y)-h(x^\star))\,dy.
\end{align*}
Since $h$ is $C^1$, we obtain that $\int_{x^\star}^{y_n}(h(y)-h(x^\star))\,dy=O((y_n-x^\star)^2)$.  Hence
$\mu_Y(\tau>n)\sim \mu(Y)^{-1}(y_n-x^\star)h(x^\star)$.
By Step~2,
$\mu_Y(\tau>n)\sim \alpha^\alpha\mu(Y)^{-1}n^{-\alpha}$.
It follows from~\eqref{eq:x1} and~\eqref{eq:h} that
$\mu(Y)=\frac{2^{1-\gamma}-1}{1-\gamma}$.
Hence
	$\mu_Y(\tau>n)\sim e_\alpha n^{-\alpha}$ where 
$e_\alpha= \alpha^{\alpha}\frac{1-\gamma}{2^{1-\gamma}-1}$.
By~\eqref{eq:da}, $e_\alpha=d_\alpha g_\alpha^{-1}$.
\end{proof}

\subsection{Proof of limit theorems}
\label{sec:proof}

In this subsection we prove Theorems~\ref{thm:stable1},~\ref{thm:stable2},~\ref{thm:Levy} and~\ref{thm:SDE}.

\begin{prop} \label{prop:Kac}
\(
\BIG \int_Y\tau\,d\mu_Y
=(1-2^{\gamma-1})^{-1}
\) 
for $\gamma<1$.
\end{prop}

\begin{proof}
Recall that $\mu([0,1])<\infty$ for $\gamma<1$.
Define the probability measure $\tilde\mu=\mu([0,1])^{-1}\mu$ on $[0,1]$.
Since $\tau$ is the first return to $Y$, it follows from Kac' lemma that
\[
\int_Y\tau\,d\mu_Y=\frac{1}{\tilde\mu(Y)}
=\frac{\mu([0,1])}{\mu(Y)}
=\frac{1}{1-2^{\gamma-1}}
\] 
as required.
\end{proof}

\begin{pfof}{Theorem~\ref{thm:stable1}}
Let $F:Y\to Y$ be the induced Thaler map as in Subsection~\ref{sec:F} with parameter $\gamma=\alpha^{-1}$.
Then $F$ is full-branch relative to the partition $Y_j=\{\tau=j\}$ of $Y$.
Moreover, $F$ has bounded distortion~\cite{Thaler80,Thaler00} and hence is a
Gibbs-Markov map as defined in Subsection~\ref{sec:GM}.
Note that $\tau_j$ in the statement of the theorem is precisely $\tau\circ F^j$.

Define $V:Y\to\R$, $V=d_\alpha^{-\gamma}\tau$.
Then $V$ is locally constant and $V\ge0$.
By Proposition~\ref{prop:tau},
\begin{equation} \label{eq:stable1}
\mu_Y(V>x)=\mu_Y(\tau>d_\alpha^\gamma x)
\sim e_\alpha (d_\alpha^\gamma x)^{-\alpha}
=g_\alpha^{-1} x^{-\alpha}
\end{equation}
as $x\to\infty$.
Hence we have verified the hypotheses of Theorem~\ref{thm:Igor}
with $c_1= g_\alpha^{-1}$ and $c_2=0$.
It follows that
\[
n^{-\gamma}d_\alpha^{-\gamma}\Big(\sum_{j=0}^{n-1}\tau_j \,- \, d_\alpha^\gamma a_n\Big)=
n^{-1/\alpha}\Big(\sum_{j=0}^{n-1}V\circ F^j \,- \, a_n\Big)\to_d X_{\alpha,1,1} 
\quad\text{as $n\to\infty$}.
\]

It remains to evaluate $a_n$ as defined in Theorem~\ref{thm:Igor}.
When $\alpha<1$, we have $a_n=0$.
For $\alpha>1$, 
\[
a_n=n\int_Y V\,d\mu_Y=n d_\alpha^{-\gamma}\int_Y\tau\,d\mu_Y
=n d_\alpha^{-\gamma}
(1-2^{\gamma-1})^{-1}
\] 
by Proposition~\ref{prop:Kac}.
Hence $d_\alpha^\gamma a_n=n\ell_\alpha$ completing the proof.
\end{pfof}

%%%%%%%%%%%%%%%%%%%%%%%%%%%%%%%%%%%%%%%%%%%%%%%%%%

\begin{pfof}{Theorems~\ref{thm:stable2} and~\ref{thm:Levy}}
Let $\Sigma=\{\pm1\}^\N$ denote the space of sequences
$\omega=(\omega_0,\omega_1,\omega_2,\ldots)$ with entries
$\omega_j\in\{\pm1\}$.  Let $\sigma:\Sigma\to\Sigma$
denote the one-sided shift $\sigma(\omega)=(\omega_1,\omega_2,\omega_3,\ldots)$.
Let $\lambda$ denote
the Bernoulli probability measure on $\Sigma$
with $\lambda(\omega_0=\pm1)=\frac12(1\pm\beta)$.

Now let $F:Y\to Y$ be the induced Thaler map as in Subsection~\ref{sec:F} with 
parameter $\gamma=\alpha^{-1}$.
Define $\tY=Y\times\Sigma$ and $\tF:\tY\to\tY$,
\[
\tF(y,\omega)=(Fy,\sigma\omega).
\]
The product measure $\tilde\mu=\mu_Y\times\lambda$ is
an ergodic $\tF$-invariant probability measure on $\tY$.
Define the partition $\{\tY_j^+,\,\tY_j^-,\,j\ge1\}$ of $\tY$,
where $\tY_j^{\pm}=\{(y,\omega):y\in Y_j,\,\omega_0=\pm1\}$.
Again $\tF$ is full-branch with
bounded distortion~\cite{Thaler80,Thaler00} and hence is a
Gibbs-Markov map as defined in Subsection~\ref{sec:GM}.

Define the locally constant observable
\[
V:\tY\to\R, \qquad
V(y,\pm1)=\pm d_\alpha^{-\gamma}\tau.
\]
Then
\[
\SMALL
\tilde\mu\big((y,\omega):V(y)>x\big)=
\tilde\mu\big((y,\omega):\omega_0=1,\,\tau(y)>d_\alpha^\gamma x\big)
=
\lambda(\omega_0=1)\mu_Y(\tau>d_\alpha^\gamma x).
\]
Hence by~\eqref{eq:stable1},
\[
\tilde\mu(V>x)\sim c_1 x^{-\alpha}, \qquad
c_1=\tfrac12(1+\beta) g_\alpha^{-1} 
\]
as $x\to\infty$.
Similarly,
\[
\tilde\mu(V<-x)\sim c_2 x^{-\alpha}, \qquad 
c_2=\tfrac12(1-\beta) g_\alpha^{-1},
\]
as $x\to\infty$ and
we obtain
\[
c_1+c_2=g_\alpha^{-1},
\qquad \frac{c_1-c_2}{c_1+c_2}=\beta.
\]
Hence it follows from Theorem~\ref{thm:Igor} that
\[
n^{-\gamma}d_\alpha^{-\gamma}\Big(\sum_{j=0}^{n-1}\delta_j\tau_j\,-\,d_\alpha^\gamma a_n\Big)
=
n^{-1/\alpha}\Big(\sum_{j=0}^{n-1}V\circ \tF^j\,-\,a_n\Big)\to_d X_{\alpha,1,\beta}
\quad\text{as $n\to\infty$}.
\]
When $\alpha<1$ we have $a_n=0$.  For $\alpha>1$,
\begin{align*}  
a_n= n\int_\tY V\,d\tilde\mu 
=n d_\alpha^{-\gamma}\beta\int_Y \tau \,d\mu_Y
 =n d_\alpha^{-\gamma}\beta 
(1-2^{\gamma-1})^{-1}
\end{align*}
by Proposition~\ref{prop:Kac}.
Hence $d_\alpha^\gamma a_n=n\beta\ell_\alpha$ completing the proof of
Theorem~\ref{thm:stable2}.

Theorem~\ref{thm:Levy} is now an immediate consequence of Theorem~\ref{thm:TK}.
\end{pfof}

\begin{pfof}{Theorem~\ref{thm:SDE}}
We verify the hypotheses of~\cite[Theorem~2.6]{ChevyrevEtAl19}.
The beginning of the proof is similar to the proof of
Theorem~\ref{thm:stable2}.
Define the induced observable
\[
V:\tY\to\R, \qquad V(y,\omega)=\sum_{j=0}^{\tau(y)-1}v^{(j)}(y,\omega).
\]
Then
\begin{align*}
V(y,\pm1) & =
\pm\eta d_\alpha^{-\gamma}(1-2^{\gamma-1})^{-\gamma}
\big((1-2^{1-\gamma})^{-1}+(\tau-1)\big)
\\ & =\pm\eta d_\alpha^{-\gamma}(1-2^{\gamma-1})^{-\gamma}
\big(\tau-(1-2^{\gamma-1})^{-1}\big).
\end{align*}
This differs from the observable $V$ in the proof of
Theorem~\ref{thm:stable2} in that $V$ is already centred and there is an extra factor of $\eta (1-2^{\gamma-1})^{-\gamma}$.
Hence by Theorem~\ref{thm:Igor},
$n^{-\gamma}\sum_{j=0}^{n-1}V\circ\tF^j\to_d
\eta (1-2^{\gamma-1})^{-\gamma} X_{\alpha,1,\beta}$.
By Remark~\ref{rmk:eta},
$n^{-\gamma}\sum_{j=0}^{n-1}V\circ\tF^j\to_d
(1-2^{\gamma-1})^{-\gamma} X_{\alpha,\eta,\beta}$.

Next, define the induced process 
$\tW_n(t)=n^{-\gamma}\sum_{j=0}^{\lfloor nt\rfloor -1}V\circ\tF^j$.
It is immediate from Theorem~\ref{thm:TK} that
$\tW_n\to_w 
(1-2^{\gamma-1})^{-\gamma} W_{\alpha,\eta,\beta}$ in $D[0,\infty)$
in the $\cJ_1$ topology (and hence in the $\cM_1$ topology).

We now apply~\cite[Theorem~2.2]{MZ15} with $B(n)=n^{-\gamma}$.  The  technical assumption~(2.2) in~\cite{MZ15} holds for all intermittent maps, including Thaler maps, by the argument in~\cite[Section~4]{MZ15}.
It follows from~\cite[Theorem~2.2 and Remark~2.3]{MZ15} and the convergence result for $\tW_n$ that
$W_n\to_w (\int_Y \tau\,d\mu_Y)^{-\gamma}(1-2^{\gamma-1})^{-\gamma} W_{\alpha,\eta,\beta}$ in $D[0,\infty)$
in the $\cM_1$ topology.
By Proposition~\ref{prop:Kac}, $W_n\to_w W_{\alpha,\eta,\beta}$.
This is the first hypothesis of~\cite[Theorem~2.6]{ChevyrevEtAl19}.

The remaining hypothesis of~\cite[Theorem~2.6]{ChevyrevEtAl19} concerns tightness in $p$-variation.  
Recall that
$\tF$ is Gibbs-Markov and the return time $\tau\ge1$ 
satisfies $\mu_Y(\tau>n)\sim {\rm const.}\, n^{-\alpha}$.
In particular, $\tau$ is regularly varying with index $\alpha$.  
Hence the desired tightness in $p$-variation is a consequence of~\cite[Theorem~6.2]{ChevyrevEtAl19}.  
This completes the proof.
\end{pfof}

%%%%%%%%%%%%%%%%%%%%%%%%%%%%%%%%%%%%%%%%%%%%%%%%%%

\section{Discussion and outlook}
\label{sec:disc}
In this paper, we designed a conceptually new method, based on homogenisation theory, for numerically simulating SDEs driven by L\'evy noise. Rather than employing a direct form of discretisation of the SDE using Taylor-expansion as done in Euler-Maruyama type discretisations, we view a continuous-time SDE as a limit of deterministic fast-slow maps. This is achieved by applying statistical limit theorems to judiciously chosen observables of intermittent Pomeau-Manneville maps. In particular, we used the intermittent Thaler map for which calculations can be done analytically. Using an induced version of the Thaler map, we deterministically generated stable laws and the associated L\'evy processes for any user-specified parameters $\alpha \in(0,1)\cup(1,\infty)$, $\eta$ and $\beta$. For the numerical approximation of SDEs driven by L\'evy processes with $\alpha\in(1,2)$, we considered limits of suitable fast-slow maps where the fast dynamics is a non-induced Thaler map. We provide rigorous proofs employing recent statistical limit laws and deterministic homogenisation theory for the convergence of our methods. 

Our method is particularly designed to deal with Marcus SDEs with non-Lipschitz drift and diffusion terms. We showed in numerical examples that our approach is able to reproduce the statistics of $\alpha$-stable laws and $\alpha$-stable L\'evy processes as well as of SDEs. Moreover, going beyond the theory, our numerical treatment of Marcus SDEs was able to reproduce the stationary density as well as capture temporal statistics in the form of the auto-correlation function. In our numerical examples we considered one-dimensional Marcus SDEs with multiplicative noise that is exact in the sense that a change of coordinates leads to an additive noise structure for the transformed SDE. Our second example showed that although additive noise SDEs are in principle amenable to Euler-Maruyama type discretisations, this may lead to false results when there are natural boundaries. The usefulness of our fast-slow map approximation will be even more evident in the setting of multi-dimensional Marcus SDEs with non-Lipschitz drift and diffusion terms, where typically a change of coordinates cannot lead to a transformed system with additive noise structure, making Euler-Maruyama discretisations much less straightforward. 

Our strategy to approximate SDEs by deterministic fast-slow maps is not restricted to SDEs driven by L\'evy noise. Unbounded increments also occur for SDEs driven by Brownian motion and non-Lipschitz drift and diffusion terms similarly pose well known limitations for traditional discretisation schemes. Homogenisation theory for deterministic fast-slow systems with strongly chaotic dynamics leading to SDEs on the diffusive time scale driven by Brownian motion is well 
developed \cite{Dolgopyat04,Dolgopyat05,GottwaldMelbourne13c,KellyMelbourne16,CFKMZsub}
and can be applied along the lines pursued here.
The equivalent of the Marcus integral for SDEs driven by Brownian motion is the Stratonovich integral, preserving classical calculus. However, in the case of Brownian motion, the fast-slow maps typically generate corrections to the drift terms which are neither It\^o nor Stratonovich (see for example \cite{GivonKupferman04,MacKay10,GottwaldMelbourne13c,KellyMelbourne16,FrankGottwald18}). In principle, these additional terms could be accounted for by introducing modified drift terms in the fast-slow map, but such terms involve correlation functions and would require computationally costly estimations. Hence, the power of our approach which uses analytic calculations when designing the appropriate fast-slow maps, really lies within the realm of SDEs driven by L\'evy noise.\\

The computational cost of our method depends on the value of $\varepsilon$ required for sufficient convergence: to evolve the dynamics to time $t=1$ $n=1/\eps$ iterations of the map are required. This is to be compared with the Euler-Maruyama method which requires $n=1/\Delta t$ iterations. What time step $\Delta t$ or what value of $\eps$ would be necessary depends on the SDE under consideration. Currently our theory does not provide convergence rates which would allow to better assess the required computational cost. The numerical examples provided in Section~\ref{sec:numerics} and \ref{sec:numerics2}, however, are promising.\\

We make a final remark on the general approach taken in this work of unravelling a stochastic differential equation into a deterministic multi-scale system, which may seem counter-intuitive to the scientist who views SDEs as reduced systems of complex multi-scale deterministic systems. By passing from deterministic multi-scale dynamics to an SDE representing the slow variables, modellers gain (amongst other things) the numerical advantage of avoiding to have to deal with resolving stiff multi-scale dynamics and hence needing to apply prohibitively small time steps. This has been one of the many reasons to resort to stochastic parameterisations as applied in molecular dynamics and in climate science \cite{LeimkuhlerMatthews,GottwaldEtAl17}. Here we go in the opposite direction. The issue of stiffness, however, does not arise as we work directly within the framework of maps whereas modellers consider continuous time multi-scale systems which must then be discretised with all the associated numerical issues. 

%%%%%%%%%%%%%%%%%%%%%%%%%%%%%%%%%%%%%%%%%%%%%%%%%%

 \paragraph{Acknowledgements}
 GAG would like to thank Giles Vilmart for bringing the problem of using Euler-Maruyama for non-Lipschitz SDEs to us, and also for many enlightening discussions. GAG and IM would also like to thank Ben Goldys for helpful discussions. We would like to thank John Nolan for generously sharing his software package STABLE with us. This research began during an International Research Collaboration Award at the University of Sydney, and support continued via the visitor program at the Sydney Mathematical Research Institute. The research of IM was supported in part by a European Advanced Grant StochExtHomog (ERC AdG 320977).

%%%%%%%%%%%%%%%%%%%%%%%%%%
%\bibliographystyle{plain}
%\bibliography{Thaler}

\end{document}